\documentclass[11pt]{article}

\usepackage{amsfonts}
\usepackage{amscd}
\usepackage{amssymb}
\usepackage{amsthm}
\usepackage{amsmath, xspace}
\usepackage{blkarray}
\usepackage{bm}
\usepackage{cancel}
\usepackage{color}
\usepackage{graphics}
\usepackage{graphicx}
\usepackage{enumitem}
\usepackage{fancyhdr}
\usepackage{mathdots}
\usepackage{mathrsfs}
\usepackage{multicol}
\usepackage{stmaryrd}
\usepackage{ytableau}
\usepackage[all]{xy}
\usepackage{dsfont}
\usepackage[plainpages,backref]{hyperref}

\usepackage{lscape}

\theoremstyle{plain}
\newtheorem{theorem}{Theorem}[section]
\newtheorem{lemma}[theorem]{Lemma}
\newtheorem{proposition}[theorem]{Proposition}
\newtheorem{corollary}[theorem]{Corollary}
\theoremstyle{definition}
\newtheorem*{definition}{Definition}

\newtheorem{problem}{Problem}

\theoremstyle{remark}

\numberwithin{equation}{section}

\usepackage{ytableau}
\DeclareMathOperator{\ch}{ch}
\DeclareMathOperator{\Res}{Res}
\DeclareMathOperator{\Ind}{Ind}

\usepackage{tikz}
\usepackage{float}

\title{On Minimal Polynomials of Elements\\in Symmetric and Alternating Groups}
\author{Velmurugan S\footnote{The Institute of Mathematical Sciences, Chennai, Homi Bhabha National Institute, Mumbai.\\E-mail: velmurugan@imsc.res.in}}
\date{\vspace{-5ex}}
\begin{document}
	\maketitle
	\begin{abstract}
         Let $ (\rho, V) $ be an irreducible representation of the symmetric group $ S_n$ (or the alternating group $ A_n$), and let $ g $ be a permutation on $n$ letters with each of its cycle lengths divides the length of its largest cycle.
         We describe completely the minimal polynomial of $\rho(g)$, showing that, in most cases, it equals $x^{o(g)} - 1 $, with a few explicit exceptions. As a by-product, we obtain a new proof (using only combinatorics and representation theory) of a theorem of Swanson that gives a necessary and sufficient condition for the existence of a standard Young tableau of a given shape and major index $r \ \text{mod} \ n$, for all $r$. Thereby, we give a new proof of a celebrated result of Klyachko on Lie elements in a tensor algebra, and of a conjecture of Sundaram on the existence of an invariant vector for $n$-cycles.  We also show that for elements $g$ in $S_n$ or $A_n$ of even order, in most cases, $\rho(g)$ has eigenvalue $-1$, with a few explicit exceptions.

                
			
 	\end{abstract}



	\section{Introduction}


Let $G$ be the symmetric group $S_n$ or the alternating group $A_n$.
For $g\in G$, let $o(g)$ denote the order of $g$.
In this paper, we will compute the minimal polynomial of the operator $\rho(g)$, where $\rho$ is an irreducible representation of $G$ and $g\in G$ is such that all its cycle lengths divide the length of its largest cycle.
\begin{theorem}\label{Theorem:main}
    For partitions $\lambda$ and $\mu$ of $n$, let $(\rho_\lambda,V_\lambda)$ denote the irreducible representation of $S_n$ corresponding to $\lambda$ and $w_\mu\in S_n$ denote a permutation with cycle type $\mu$.
    Assume that each part of $\mu$ divides its largest part.
    Then the minimal polynomial $p(x)$ of $\rho_\lambda(w_\mu)$ is $x^{o(w_\mu)}-1$ except in the following cases:
     \begin{enumerate}
        \item $\lambda=(n)$, $\mu\neq (1^n)$, where $p(x)=(x-1)$,
         \item $\lambda=(n-1,1)$, $\mu=(n)$, where $p(x)=\tfrac{x^n-1}{x-1}$,
         \item $\lambda=(2,1^{n-2})$, $\mu=(n)$, where $p(x)=\tfrac{x^n-1}{x-(-1)^{n-1}}$,
         \item $\lambda=(1^n)$, $\mu\neq (1^n)$, where $p(x)=(x-(-1)^{n-\mathrm{len}(\mu)})$,
         \item $\lambda = (3, 3)$, $\mu=(6)$, where $p(x)=\tfrac{x^6-1}{(x-\zeta)(x-\bar{\zeta})}$ and $\zeta$ is a primitive $6$th root of unity,
         \item $\lambda = (2, 2, 2)$, $\mu=(6)$, where $p(x)=\tfrac{x^6-1}{(x-\zeta)(x-\bar{\zeta})}$ and $\zeta$ is as above,
         \item $\lambda=(2,2)$, $\mu=(4)$, where $p(x)=\tfrac{x^4-1}{(x-i)(x+i)}$ and $i$ is a primitive $4$th root of unity,
         \item $\lambda=(2,2)$, $\mu=(3,1)$, where $p(x)=\tfrac{x^3-1}{x-1}$,
         \item $\lambda=(2,2)$, $\mu=(2,2)$, where $p(x)=\tfrac{x^2-1}{x+1}$.
     \end{enumerate}
\end{theorem}

The corresponding theorem for the alternating group $ A_n$ is as follows:
\begin{theorem}\label{Theorem:Main_Alt}
    Let $(\rho,V)$ be an irreducible representation of the alternating group $ A_n$
     and let $g\in  A_n$ with cycle type $\mu$.
     Suppose that every part of $\mu$ divides its largest part.
     Then the minimal polynomial of $\rho(g)$ is $x^{o(g)}-1$ except in the following cases:
     \begin{enumerate}
         \item $V=V_{(n)}$,  $\mu\neq (1^n)$ where $p(x)=x-1$,
         \item $V=V_{(n-1,1)}$, $\mu=(n)$ and $n\geq 5$ odd, where $p(x)=\tfrac{x^n-1}{x-1}$,
         \item $V=V_{(3,1,1)}^+$, $\mu=(5)^+$ or $V=V_{(3,1,1)}^-$, $\mu=(5)^-$, where $p(x)=\tfrac{x^5-1}{(x-\zeta^2)(x-\zeta^3)}$ with $\zeta=\exp(2\pi i/5)$,
         \item $V=V_{(3,1,1)}^+$, $\mu=(5)^-$ or $V=V_{(3,1,1)}^-$, $\mu=(5)^+$, where $p(x)=\tfrac{x^5-1}{(x-\zeta)(x-\zeta^4)}$ with $\zeta$ as above,
         \item $V=V_{(2,2)}^+$, $\mu=(3,1)^+$ or $V=V_{(2,2)}^-$, $\mu=(3,1)^-$,  where $p(x)=x-\exp(2\pi i/3)$,
         \item $V=V_{(2,2)}^+$, $\mu=(3,1)^-$ or $V=V_{(2,2)}^-$, $\mu=(3,1)^+$,  where $p(x)=x-\exp(4\pi i/3)$,
         \item $V=V_{(2,2)}^\pm$, $\mu=(2,2)$, where $p(x)=x-1$.
     \end{enumerate}
\end{theorem}
A. E. Zalesskii posed the following problem, hoping that it would have a positive answer. We refer the reader to~\cite{10.1007/BFb0082027} for more related problems.

For $G=GL_n(\mathbb{F}_{p^k})$, let $g\in G$ and $\phi$ be an irreducible representation over $\mathbb{C}$, the field of complex numbers or over an algebraically closed field of characteristic $q\neq p$.
\begin{problem}[Zalesskii \cite{10.1007/BFb0082027}]\label{problem:zalesski}
    Describe the triples $(G,\phi,g)$ such that $1<\deg(\phi(g))<m(g)$,
where $m(g)$ is the order of the element $gZ(G)$ in $G/{(Z(G))}$, and $\deg(\phi(g))$ denotes the degree of the minimal polynomial of $\phi(g)$.
\end{problem}
Note that, in general, $\deg(\phi(g))\leq m(g)$.

Zalesskii solved Problem~\ref{problem:zalesski} when $g$ has order $p$ in \cite{Zalesski_in_russian}.
Zalesskii~\cite{Zalesski_quasi_simple} determined the irreducible representations of quasi-simple groups in which some element of prime order $p$ has minimal polynomial whose degree is less than $p$. For more results in this direction, see~\cite{Zalesski_distinct_eigen},\cite{Zalesski_prime_projective_alternating},\cite{Zalesski_quasi_cyclic_sylow}.

Our Theorems~\ref{Theorem:main},~\ref{Theorem:Main_Alt} solve the following problem of Tiep and Zalesskii for the symmetric and alternating groups over the field of complex numbers.

\begin{problem}\cite[Problem 1.1]{Tiep_Zalesski}\label{quesion:Tiep_zalesski}
	Determine all possible values for $\deg(\theta(g))$, and if possible, all triples
	$(G, \theta, g)$ with $\deg(\theta(g)) < m(g)$, in the first instance under the condition that $o(g)$ is a $p$-power, for some prime $p$.
\end{problem}

This work is motivated by one more interesting problem.
Before defining the problem, let us introduce some notation.

For a partition $\lambda$ of $n$, let $\mathrm{SYT}(\lambda)$ denote the set of all standard Young tableau with shape $\lambda$ and the entries belong to $\{1,2,\dotsc,n\}$. The \emph{major index} of a standard tableau (denoted $\mathrm{maj}(T)$) $T$ is the sum of all $i\in \{1,2,\dots,n-1\}$ such that the row of $T$ containing $i+1$ is below the row of $T$ containing $i$.

Then let
\begin{displaymath}
    a_\lambda^r=|\{T\in \mathrm{SYT}(\lambda)\mid \mathrm{maj}(T)\equiv r \text{ mod } n\}|,
\end{displaymath}
where $r$ is taken modulo $n$.

Kra\'skiewicz and Weyman~\cite{Kraskiewicz_Weyman} related the above number $a^r_\lambda$ to the representation theory of symmetric groups as follows:
\begin{theorem}\cite[Theorem 1]{Kraskiewicz_Weyman}\label{Theorem:Kra_wey}
Let $C_n$ be the cyclic subgroup of $S_n$ generated by $w_n=(1 2\dotsc n).$
    Let $\delta^r$ be an irreducible character of the $C_n$ obtained by sending $w_n$ to $e^{2\pi ir/n}$.
    Then
    \begin{displaymath}
        a_{\lambda}^r=\langle \Ind_{C_n}^{S_n}\delta^r, \chi_\lambda\rangle_{S_n}  =  \langle \delta^r, \Res_{C_n}^{S_n}\chi_\lambda\rangle_{C_n}.
    \end{displaymath}
\end{theorem}
Note that $a_\lambda^r=a_\lambda ^k$ whenever we have an equality $(r,n)=(k,n)$ of greatest common divisors, since characters of the symmetric group are real-valued.

For a vector space $V$, consider the tensor algebra $TV$ of $V$.
Then $TV=\bigoplus_{n\geq 0} T^nV$ and each $T^nV$ is an $GL(V)$-module.
We recall that the irreducible representations of $GL(V)$ are indexed by partitions with at most $\dim(V)$ parts, and the representation corresponding to a partition $\lambda$ is denoted by $W_\lambda$. The irreducible constituents of $T^nV$, as an $GL(V)$-module, are those $W_\lambda$ for which $\lambda$ is a partition of $n$ with at most $\dim(V)$ parts. 
Since $TV$ is an associative algebra, it is, in particular, a Lie algebra.
Consider the Lie subalgebra $LV$ of $TV$ generated by $V$.
Then it turns out that $LV$ is the free Lie algebra on $V$ (see \cite{Free_Lie_Reutenauer}, Theorem $0.4$).
We have $LV=\bigoplus_{n\geq 1} L_nV$, where $L_nV=LV\cap T^nV$.
It is easy to see that each $L_nV$ is also a $GL(V)$-module.
It is natural to ask what the irreducible constituents of $L_nV$ are.
The following two theorems of Klyachko answer this question.
Firstly, he beautifully related this question to a question about the symmetric group as follows:
\begin{theorem}\cite[Corollary 1]{Klyachko}\label{theorem:Klyachko_gen_sym}
    Let $\lambda$ be a partition of $n$ with at most $\dim(V)$ parts. Then 
    we have 
    \begin{displaymath}
        m(W_\lambda,L_nV)=m(V_\lambda,\Ind_{C_n}^{S_n} w), 
    \end{displaymath}
    where $m(W_\lambda,L_nV)$ denotes the multiplicity of a $GL(V)$-module $W_\lambda$ in $L_nV$ and $m(V_\lambda,\Ind_{C_n}^{S_n} w)$ denotes the multiplicity of a $S_n$-module $V_\lambda$ in $\Ind_{C_n}^{S_n} w$ for any faithful one dimensional character $w$ of the cyclic group $C_n$.
\end{theorem}
He solved the problem for the symmetric group as follows:
\begin{theorem}\cite[Proposition 2]{Klyachko}\label{theorem:Klyachko_sym}
Let $\lambda$ be a partition of $n\geq 3$. Then the restriction to $C_n$ of every irreducible representation $V_\lambda$ of $S_n$ contains a faithful representation of $C_n$ except when $\lambda=(n),(1^n),(2,2)$ or $(2,2,2)$.
\end{theorem} 

Combining the above Theorems~\ref{theorem:Klyachko_gen_sym},~\ref{theorem:Klyachko_sym}, Klyachko showed the following:
\begin{theorem}\label{theorem:Klyachko_gen_lie_rep}
Let $\lambda$ be a partition of $n\geq 7$ having at most $\dim(V)$ parts.
Then every irreducible representation $W_\lambda$ of $GL(V)$ which appears as an irreducible constituent of $T^nV$, is also appears as an irreducible constituent of $L_nV$ except when $\lambda=(n),(1^n)$. Equivalently, $W_\lambda$ is an irreducible constituent of the $GL(V)$-module $L_nV$ except when $\lambda=(n),(1^n)$.
\end{theorem}
For a partition $\lambda\vdash n$ with at most $\dim(V)$ parts, we have, using Theorems~\ref{Theorem:Kra_wey},~\ref{theorem:Klyachko_gen_sym},
\begin{equation}
    m(W_\lambda,L_nV)=m(V_\lambda,\Ind_{C_n}^{S_n} w)=a_\lambda^1.
\end{equation}
Klyachko's proof of Theorem~\ref{theorem:Klyachko_sym} simply uses representation theory.
The goal is to prove the corresponding result for all $r\geq1$ using representation theory and combinatorics, which we accomplish in this article. We provide a new proof of Theorem~\ref{theorem:Klyachko_sym} using Theorem~\ref{Theorem:main}. Using our solution with Theorem~\ref{theorem:Klyachko_gen_sym}, we get a new proof of a celebrated Klyachko's theorem~\ref{theorem:Klyachko_gen_lie_rep} on Lie representations of general linear group $GL(V)$.
Using Theorems~\ref{Theorem:Kra_wey},~\ref{theorem:Klyachko_sym}, we have that $a_\lambda^1>0$ except when $\lambda=(n),(1^n),(2,2),(2,2,2)$. For a Combinatorial proof of this result, see M. Johnson~\cite{M.Johnson}.
For more proofs of Theorem~\ref{theorem:Klyachko_gen_lie_rep}, see~\cite{Kovacs_Stohr}, \cite{Schocker_kly}.

Our Theorem~\ref{Theorem:main} yields the following for the permutations with cycle type equal to $(n)$:
\begin{theorem}\label{theorem:Main_min}
    Let $\lambda$ be a partition of $n$. The minimal polynomial $p(x)$ of $\rho_\lambda(w_n)$ is $x^n-1$ except in the following cases:
    \begin{enumerate}
        \item $\lambda=(n)$ with $n>1$ where $p(x)=(x-1)$,
         \item $\lambda=(n-1,1)$ where $p(x)=\tfrac{x^n-1}{x-1}$,
         \item $\lambda=(2,1^{n-2})$ where $p(x)=\tfrac{x^n-1}{x-(-1)^{n-1}}$,
         \item $\lambda=(1^n)$ where $p(x)=(x-(-1)^{n-1})$,
         \item $\lambda = (3, 3)$ where $p(x)=\tfrac{x^6-1}{(x-\zeta)(x-\bar{\zeta})}$ and $\zeta$ is a primitive $6$th root of unity,
         \item $\lambda = (2, 2, 2)$ where $p(x)=\tfrac{x^6-1}{(x-\zeta)(x-\bar{\zeta})}$ with $\zeta$ as above,
         \item $\lambda=(2,2)$ where $p(x)=\tfrac{x^4-1}{x^2+1}$.
    \end{enumerate}
\end{theorem}

This result was proved by Swanson\cite{Swanson} and independently by Yang and Staroletov~\cite{Yang_Staroletov} (who proved a more general result). Their proofs use asymptotics of characters of symmetric groups and number theory.
We provide a new proof (\ref{proof_of_Theorem:main_min}) of this result using only representation theory and combinatorics, which, we believe, is more accessible.
We remark that a purely combinatorial proof of the above result is still an open problem.

We would also like to note that our main theorem~\ref{Theorem:main} reconfirms the following conjecture of Sundaram, which was proved by Swanson, as a special case of Theorem \ref{theorem:Main_min}.
\begin{theorem}~\cite[Remark 4.8]{Sundaramconjecture}\label{conjecture:Sundaram}
    Let $\lambda$ be a partition of $n$. Then $1$ is an eigenvalue of the operator $\rho_\lambda(w_n)$ (equivalently, $\rho_\lambda(w_n)$ has a nonzero invariant vector in $V_\lambda$) except in the following cases.
    \begin{enumerate}
        \item $\lambda=(n-1,1)$ with $n\geq 2$,
        \item $\lambda=(2,1^{n-2})$ with $n\geq 3$ odd,
        \item $\lambda=(1^n)$ with $n$ even.
    \end{enumerate}
\end{theorem}

For an element $g$ in a group $G$, we say that $g$ has a nonzero invariant vector in an irreducible representation $(\rho,V)$ of $G$ if there exists a nonzero vector $v\in V$ such that $\rho(g)v=v.$

Together with Amrutha P and A. Prasad, we have shown the existence of a nonzero invariant vector when $G$ is the symmetric group, $g$ is any permutation and $V$ is any irreducible representation of $S_n$.
\begin{theorem}\cite[Main theorem]{Inv_vectors}\label{theorem:inv_vec_sym}
    Let $\lambda,\mu$ be partitions of $n$. Then $w_\mu$ has a nonzero invariant vector in $V_\lambda$ (equivalent to $1$ is an eigenvalue of $\rho_\lambda(w_\mu)$) except in the following cases:
    \begin{enumerate}
			\item $\lambda=(1^n)$, $\mu$ is any partition of $n$ for which $w_\mu\notin A_n$,
			\item $\lambda=(n-1,1)$, $\mu=(n)$, $n\geq 2$,
			\item $\lambda=(2,1^{n-2})$, $\mu=(n)$, $n\geq 3$ is odd,
			\item $\lambda=(2^2,1^{n-4})$, $\mu=(n-2,2)$, $n\geq 5$ is odd,
			\item $\lambda=(2,2)$, $\mu=(3,1)$,
			\item $\lambda=(2^3)$, $\mu=(3,2,1)$,
			\item $\lambda=(2^4)$, $\mu=(5,3)$,
			\item $\lambda=(4,4)$, $\mu=(5,3)$,
			\item $\lambda=(2^5)$, $\mu=(5,3,2)$.
		\end{enumerate}
\end{theorem}

The proof is combinatorial, except for the base case, which is Theorem~\ref{theorem:Main_min}.
By proving Theorem~\ref{theorem:Main_min} using representation theory and combinatorics, we complete the proof of Theorem~\ref{theorem:inv_vec_sym} within the same framework.

Similarly, for alternating groups, we proved, with Amrutha P and A. Prasad, the following result:
\begin{theorem}\cite[Theorem C]{p2024cyclic} \label{Inv_vectors_Alternating}
     Let $\mu$ be a partition of $n$. Then $1$ is an eigenvalue of $w_\mu$ in an irreducible representation $V$ of the alternating group $ A_n$ except in the following cases.
    \begin{enumerate}
        \item $V=V_{(n-1,1)}$ and  $\mu=(n)$ with $n>3$ is odd.
        \item $V=V_{(2,1)}^\pm$ and $\mu=(3)$,
        \item $V=V_{(2,2)}^\pm$ and $\mu=(3,1)$,
        \item $V=V_{(4,4)}$ and $\mu=(5,3)$.
    \end{enumerate}
\end{theorem}
The proof is combinatorial except the base case which is the following Theorem~\ref{theorem:Main_2_min_alternating} which is a special case of \ref{Theorem:Main_Alt}.
By proving Theorem~\ref{theorem:Main_2_min_alternating} using representation theory and combinatorics, we complete the proof of Theorem~\ref{Inv_vectors_Alternating} within the same framework.
\begin{theorem}\label{theorem:Main_2_min_alternating}
    Let $n$ be a positive odd integer. Then the minimal polynomial $p(x)$ of $w_n$ in an irreducible representation $V$ of the alternating group $ A_n$ has degree less than $n$ in precisely the following cases.
    \begin{enumerate}
        \item $V=V_{(n)}$, where $p(x)=x-1$,
        \item $V=V_{(n-1,1)}$, where $p(x)=\tfrac{x^n-1}{x-1}$,
        \item $V=V_{(3,1,1)}^\pm$, where $p(x)=\tfrac{x^n-1}{(x-\zeta)(x-\Bar{\zeta})}$, where $\zeta$ is a primitive $5$th root of unity.
    \end{enumerate}
\end{theorem}

Finally, we prove the counterpart of Theorem~\ref{theorem:inv_vec_sym}, which is of independent interest.
\begin{theorem}\label{theorem:eigen_value_negative_universal}
    Suppose that $w_\mu$ has even order. Then there is a nonzero vector $v$ in $V_\lambda$ such that $\rho_\lambda(w_\mu)v=-v$ (equivalently, $-1$ is a eigenvalue of $\rho_\lambda(w_\mu)$) except in the following cases.
    \begin{enumerate}
        \item $\lambda=(n)$,
        \item $\lambda=(n-2,2)$ and $\mu=(n-2,2)$ with $n\geq 5$ odd,
        \item $\lambda=(2,1^{n-2})$ and $\mu=(n)$ with $n\geq 4$ even,
        \item $\lambda=(1^n)$ and $\mu$ is such that $w_\mu$ is an even permutation,
        \item $\lambda=(2,2)$ and $\mu=(2,2)$.
        \item $\lambda=(3,3)$ and $\mu=(3,2,1)$,
        \item $\lambda=(5,5)$ and $\mu=(5,3,2)$.
    \end{enumerate}
\end{theorem}

The analogous theorem for the alternating group is as follows:
\begin{theorem}\label{theorem:eigen_value_negative_universal_Alt}
    Let $w_\mu$ be an element in the alternating group $A_n$ with cycle type $\mu$ and suppose that $w_\mu$ has an even order. Then for every non-linear irreducible representation $(\rho,V)$, there exists a nonzero vector $v$ in $V$ such that $\rho(w_\mu)v=-v$ except when $V=V_{(2,2)}^\pm$.
\end{theorem}


The article is organized as follows: In Sect.~\ref{section:Notation and preliminaries}, we introduce the notation and give basic definitions. In Sect.~\ref{section:Ingredients for the proof of the main theorems}, we provide some important lemmas and a proposition to prove Theorem~\ref{theorem:Main_min}. In Sect.~\ref{section:Proof of Thorem theorem:Main_min}, we prove Theorem~\ref{theorem:Main_min}. In Sect.~\ref{section:Proof of Theorem Theorem:main}, we prove Theorems \ref{Theorem:main}, \ref{theorem:eigen_value_negative_universal}, \ref{theorem:eigen_value_negative_universal_Alt}.
    
\section{Notation and Preliminaries}\label{section:Notation and preliminaries}
Let $G$ be a finite group and $H$ be its subgroup.
Let $Irr(G)$ denote the set of all irreducible characters of $G$.
For an irreducible character $\Psi$ of a subgroup $H$, let $Irr(G|\Psi)=\{\chi\in Irr(G)|\langle \Res^G_H \chi,\Psi\rangle_H\neq 0\}$.
Let $(\chi,V)$ and $(\Psi,W)$ be representations of $G$.
If $(\chi,V)$ is a sub-representation of $(\Psi,W)$, then we denote this by $\chi\leq\Psi$ or $V\leq W$.

For a composition $\alpha$ of $n$, we define a natural permutation $w_\alpha$ of $[n]$ as follows:
\begin{displaymath}
    w_\alpha=(1~2~\dotsc~\alpha_1)(\alpha_1+1~\alpha_1+2~\dotsc~\alpha_1+\alpha_2)\dotsc(\sum_{i=1}^{k-1}\alpha_i+1 ~\sum_{i=1}^{k-1}\alpha_i+2~\dotsc~\sum_{i=1}^{k}\alpha_i=n),
\end{displaymath}
where $k$ is the length of $\alpha$.
Let $C_\alpha$ denote the cyclic subgroup generated by $w_\alpha$. When $\alpha=(n)$, we write $C_n$ instead of $C_{(n)}$.
Note that the cycle type of $w_\alpha$ is the partition obtained from $\alpha$ by reordering the parts of $\alpha$ in decreasing order. Hence $C_\alpha$ and $C_\mu$ are conjugate subgroups in the symmetric group $S_n$.

We make frequent use of the following observation.
Let $\delta$ be an irreducible character of $C_\alpha$ and $\tilde\delta$ be an irreducible character of $C_{\alpha_1}\times\dotsc\times C_{\alpha_k}$ such that $Res^{C_{\alpha_1}\times\dotsc\times C_{\alpha_k}}_{C_\alpha}\tilde\delta=\delta$.
Then
\begin{displaymath}
\Ind_{C_\alpha}^{S_n} \delta\geq \Ind_{C_{ \alpha_1}\times C_{ \alpha_2}\times \dotsc \times C_{ \alpha_k}}^{S_n}\tilde\delta\geq \Ind_{S_{ \alpha_1}\times S_{ \alpha_2}\times \dotsc \times S_{ \alpha_k}}^{S_n}\Ind_{C_{ \alpha_1}\times C_{ \alpha_2}\times \dotsc \times C_{ \alpha_k}}^{S_{ \alpha_1}\times S_{ \alpha_2}\times \dotsc \times S_{\alpha_k}}\tilde\delta, 
\end{displaymath} 

The representations of the symmetric group $S_n$ are indexed by partitions of $n$.
We denote the irreducible representation of $S_n$ indexed by the partition $\lambda$ by $(\rho_\lambda,V_\lambda)$ and its character by $\chi_\lambda$. For specific character values we use the recursive Murnaghan-Nakayama rule.
We shall introduce some notation before defining the rule.
We always use the English notation for the Young diagrams of partitions.
For a cell $b$ in a Young diagram of a partition $\lambda$, the hook length of $b$ is one plus the number of cells in the Young diagram which lie directly below it or directly to its right.
For a cell $b$ of the Young diagram of the partition $\lambda$, the rim hook ($\text{rim}_b$) of the cell $b=(i,j)$ is the set of cells $(u,v)$ such that $u\geq i, ~v\geq j$ and the Young diagram of $\lambda$ does not contain a cell in the position $(i+1,j+1)$. The height ($\text{ht}(\text{rimb}_b)$) of $\text{rim}_b$ is one minus the number of rows of $\lambda$ which it intersects.

Fix any part of $\mu$ (say $\mu_t$).
The fastest way of computing these character values is using the recursive Murnaghan-Nakayama rule, which is as follows:
\begin{displaymath}
    \chi_\lambda(w_\mu)=\sum\limits_{b\in \lambda \text{ with } h_b=\mu_t} (-1)^{\text{ht}(\text{rim}_b)}\chi_{\lambda-\text{rim}_b}(w_{\tilde\mu}),
\end{displaymath}
where $b$ varies over the cells $b$ of the Young diagram of $\lambda$ such that hook length of the cell $b$ is equal to $\mu_t$ and $\lambda-\text{rim}_b$ is the partition of shape obtained from the shape of $\lambda$ by removing the cells of $\text{rim}_b$.

Let us discuss the relation between the characters of symmetric groups and the symmetric functions.
For more details and proofs see~\cite{Mac_sym},\cite{EC_2_Stanley}.
Let $\Lambda$ denote the ring of symmetric functions in the variables $x_1,x_2,\dotsc$.
For a partition of $\lambda$, let $p_\lambda$, $s_\lambda$ denote the power sum and Schur symmetric functions respectively. 

Let $h$ be a class function of the symmetric group $S_n$, the Frobenius characteristic $\ch$ from the direct sum $\bigoplus_{n=1}^\infty R(S_n)$ of the class functions $R(S_n)$ on $S_n$ is defined as follows:
\begin{align*}
    \ch : \bigoplus_{n\geq 1} R(S_n)\rightarrow \Lambda \\
    g\mapsto\sum_{\mu\vdash n} \tfrac{h(w_\mu)p_\mu}{z_\mu},
\end{align*}
where $z_\mu=\prod_i i^{m_i}m_i!$ with $m_i$ being the number of parts of $\mu$ equal to $i$.
In particular, it is known that the Frobenius characteristic of $\chi_\lambda$ is the Schur function $s_\lambda$.
Then $\ch$ is a ring homomorphism in the following sense:
\begin{displaymath}
    \ch\Ind_{S_m\times S_t}^{S_{m+t}} \chi_\mu\otimes \chi_\nu=s_\mu s_\nu,
\end{displaymath}
for all partitions $\mu\vdash m$ and $\nu\vdash t$.

The RHS can be computed combinatorially by the Littlewood-Richardson rule.
\begin{definition}[Littlewood-Richardson rule]
    Let $\mu$ and $\nu$ be partitions of $m$ and $t$ respectively.
    Then
    \begin{displaymath}
        s_\mu s_\nu=\sum_\lambda c_{\mu\nu}^\lambda s_\lambda,    
    \end{displaymath}
    where $c_{\mu\nu}^\lambda$ is the number of semistandard tableaux of shape $\lambda/\mu$ and type $\nu$ such that the reverse row reading word is a lattice permutation, see \cite[Section I.9]{Mac_sym}. Also, a semistandard tableau of shape $\lambda/\mu$ and type $\nu$ such that the reverse row reading word is a lattice permutation is also called an LR-tableau of shape $\lambda/\alpha$ and type $\nu$.
\end{definition}
For symmetric function $f$ and $g$, we say that $f-g$ is Schur positive if $f-g$ is a nonnegative integer linear combination of Schur functions, equivalently $f-g$ is equal to $ch(h)$ for some character $h$.
Let $f_\mu=\Ind_{C_\mu}^{A_n}\mathbf{1}$.
If $w_\mu$ has even order, then let $g_\mu=\Ind_{C_\mu}^{S_n}-\mathbf{1}$
where $\mathbf{1}$ denotes the trivial character of $C_\mu$ and $\mathbf{-1}$ denotes the linear character obtained by sending $w_\mu$ to $-1$.
Since $C_\mu\leq C_{\mu_1}\times\dotsc\times C_{\mu_k}$, we have
\begin{align*}
    \Ind_{C_\mu}^{S_n} 1 &= \Ind_{C_{\mu_1}\times\dotsc\times C_{\mu_k}}^{S_n} \Ind_{C_\mu}^{C_{\mu_1}\times\dotsc\times C_{\mu_k}} 1 \geq \Ind_{C_{\mu_1}\times\dotsc\times C_{\mu_k}}^{S_n} 1\\
    \Ind_{C_\mu}^{S_n} 1 &\geq \Ind_{C_{\mu_1}\times\dotsc\times C_{\mu_k}}^{S_n} 1.\\
\end{align*}
Therfore, 
\begin{equation}\label{Eq:cyclic_cycles_trivial}
    f_\mu  \geq \prod_{i\geq 1}f_{\mu_i}
\end{equation}

Let $\mu\vdash n$, such that $\mu_i|\mu_1$ for all $i=1,\dotsc,k$.
Let $\eta$ be any irreducible character of $C_\mu$.
We have an irreducible character $\tilde\eta$ for $C_{\mu_1}$, obtained by mapping $w_{\mu_1}$ to $\eta(w_\mu)$.
Then 
\begin{equation}
    \Ind_{C_\mu}^{S_n} \eta \geq \Ind_{C_{\mu_1}\times\dotsc\times C_{\mu_k}}^{S_n} \tilde\eta\times1\times1\dotsc\times1
\end{equation}
Let $\lambda'$ be the partition conjugate to the partition $\lambda$.
Then we have an involution $\omega$ defined by mapping each Schur function $s_\lambda$ to $s_{\lambda'}$. It turns out that $\omega$ is a ring isomorphism.

Let us discuss the representations of the alternating groups $A_n$ briefly.
It is well known that $\Res_{A_n}^{S_n}V_\lambda=\Res_{A_n}^{S_n}V_\lambda'$ is irreducible when $\lambda\neq \lambda'.$
If $\lambda=\lambda'$, then $\Res_{A_n}^{S_n}V_\lambda$ decomposes into two non-isomorphic irreducible representations (denoted by $V_\lambda^+, V_\lambda^-$) of $A_n$. 
For two partitions $\lambda$ and $\mu$, we say that $\lambda$ is greater than $\mu$ lexicographically (denoted by $\lambda>_{rev}\mu$) if the first non-vanishing difference $\lambda_i-\mu_i$ is positive.
Then the irreducible representations of the alternating group $A_n$ are indexed by the following set 
\begin{displaymath}
\{\lambda\vdash n|\lambda\neq\lambda' \text{ and } \lambda >_{rev} \lambda'\}\cup\{\lambda^+,\lambda^-|\lambda\vdash n, \lambda=\lambda'\}.
\end{displaymath}

Let us call the first set $A$ and the second set $B$.

We denote an irreducible representation (resp. character) by $(\rho_\lambda,V_\lambda)$ (resp. $\chi_\lambda$) if it belongs to the set $A$ and by $(\rho_\lambda^\pm,V_\lambda^\pm)$ (resp. $\chi_\lambda^\pm$) if it belongs to  the set $B$.
The character values are given as follows.
Let $DOP_n$ denote the set of all partitions $\mu$ of $n$ which consist of distinct odd parts. For $\mu\in DOP_n$, we define $w_\mu^+=w_\mu$ and $w_\mu^-=ww_\mu w^{-1}$ for some $w\in S_n \setminus A_n$. In this case, we can write $\mu=(2m_1+1,2m_2+1,\dotsc,2m_k+1)$ with $m_1>m_2>\dotsc>m_k$. We define $\phi(\mu)$ to be the self-conjugate partition which is equal to $(m_1,m_2,\dotsc,m_k|m_1,m_2,\dotsc,m_k)$ in Frobenius coordinates.

For the irreducibles in the set $A$, we have $\chi_\lambda(w_\mu),$ which coincides with the
symmetric group case.
For the irreducibles in the set $B$, we have $\chi_\lambda^\pm(w_\mu)=\tfrac{\chi_\lambda(w_\mu)}{2}$ if $\phi(\mu)$ is not equal to $\lambda$, otherwise $\phi(\mu)=\lambda$ and we have $\chi_\lambda^\pm(w_\mu^+)=\tfrac{1}{2}(\epsilon_\mu\pm\sqrt{\epsilon_\mu M})$ where $\epsilon_\mu=(-1)^{\sum_{i}\tfrac{\mu_i-1}{2}}$ and $M=\prod_i \mu_i$.
Also, $\chi_\lambda^\pm(w_\mu^-)=\chi_\lambda^{\mp}(w_\mu^+)$.
For more details on the representations of the alternating groups $A_n$, see \cite{JamesKerber},\cite{Representation_Amri}.

We mention a few important things regarding wreath products of groups which we use in this paper.
For a subgroup $H\leq S_n$ and $K\in S_m$, we have a natural isomorphism (see~\cite[4.1.18]{JamesKerber}) from $H\wr K$ into $S_{mn}$ as follows:
\begin{displaymath}
    \eta: H\wr K\rightarrow S_{mn}: (f,\pi)\mapsto \left(\begin{array}{@{}*{20}{c@{\,}}} (j-1)m+i
 \\ (\pi(j)-1)m+f(\pi(j))(i)
  \\
\end{array}\right)_{1\leq i\leq m,~ 1\leq j\leq n}.
\end{displaymath}
We now have the following very important observation.
\begin{lemma}\label{lemma:mn-cycle}
Let $C_n\leq S_n$ and $C_m\leq S_m$.
Then $\eta(C_m\wr C_n)\leq S_{mn}$ contains an $mn$-cycle. 
\end{lemma}
\begin{proof}
    Let $\sigma_{mn}=( ((123\dots m),e,e\dots,e);(12\dots n) )$.
Then we have
    \begin{equation}
    \eta(\sigma_{mn}) =(1~m+1~2m+1~\dots(n-1)m+1~2~m+2~2m+2~\dots(n-1)m+2~\dots~m~2m~3m\dots nm)    .
    \end{equation}    
\end{proof}
\noindent For the representations of wreath products, see~\cite{JamesKerber}.

Let us make the following useful definition.
\begin{definition}[Character reduced to a cycle]
	Let $\chi$ be an irreducible character of the cyclic group $C_n$. Then for a divisor $k$ of $n$, we define $\chi^{(k)}$ to be the irreducible character of $C_{n/k}$ in a natural way. More precisely, $\chi^{(k)}$ is the irreducible character of $C_{n/k}$ defined by $\chi^{(k)}(w_{n/k}^i)=\chi(w_n^{ki})$ for all $i$.
\end{definition}

The following corollary of Theorem \ref{theorem:inv_vec_sym} will be helpful in the proof of the main theorems.
    \begin{corollary}\label{corollary:eigen_value_sgn}
		The only pairs of partitions $(\lambda,\mu)$ of a given integer $n$ such that $w_\mu$ does not admit a nonzero vector $v$ in $V_\lambda$ such that $\rho_\lambda(w_\mu)v=sgn(w_\mu)v$  are the following:
		\begin{enumerate}
			\item $\lambda=(n)$, $\mu$ is any partition of $n$ for which $w_\mu\notin A_n$,
			\item $\lambda=(2,1^{n-2})$, $\mu=(n)$, $n\geq 2$,
			\item $\lambda=(n-1,1)$, $\mu=(n)$, $n\geq 3$ is odd,
			\item $\lambda=(n-2,2)$, $\mu=(n-2,2)$, $n\geq 5$ is odd,
			\item $\lambda=(2,2)$, $\mu=(3,1)$,
			\item $\lambda=(3,3)$, $\mu=(3,2,1)$,
			\item $\lambda=(4,4)$, $\mu=(5,3)$,
			\item $\lambda=(2^4)$, $\mu=(5,3)$,
			\item $\lambda=(5,5)$, $\mu=(5,3,2)$.
		\end{enumerate}
	\end{corollary}
    \begin{proof}
        The proof follows from Theorem~\ref{theorem:inv_vec_sym} and the fact that $V_\lambda\otimes V_{(1^n)}=V_{\lambda'}$.
    \end{proof}

\section{Ingredients for the proof of the main theorems}\label{section:Ingredients for the proof of the main theorems}
	The following proposition is key to proving Theorem~\ref{theorem:Main_min} and is motivated by \cite[Proposition~3.5]{Giannelli_law}.
    \begin{proposition}
		\label{proposition:Giannelli}
		Let $\delta$ be an irreducible character of the cyclic subgroup $C_n$ of $ S_n$. Let $p$ be a prime and $n=mp$ with both $m>1$. Let $\mu^1, \dots, \mu^p$ be partitions of $m$, not all the same, such that for all $2 \leq i \leq p$, $\langle\Res_{C_m}^{ S_m} \chi_{\mu^i},\mathds{1}_{C_m}\rangle\neq0$ and $\langle \Res_{C_m}^{ S_m}\chi_{\mu^1},\delta^{(p)}\rangle_{C_m}\neq0$.
		Let  $\lambda\vdash n$  be such that $\chi_{\mu^1} \times\dots \times \chi_{\mu^p}$ is an irreducible constituent of $\Res_{ S_m^{\times p}}^{ S_n}\chi_\lambda$. Then  $\Ind_{C_m^{\times p}}^{C_m\wr C_p} \Tilde{\delta^{(p)}}$ is a sub-representation of  $\Res_{C_m\wr C_p}^{ S_m}\chi_\lambda,$ where  $\Tilde{\delta^{(p)}}:=\delta^{(p)}\times \mathds{1}_{C_m}\times \dotsc \times \mathds{1}_{C_m}$.
		As a consequence, we have
		\begin{displaymath}
			\langle\Res_{C_n}^{ S_m}\chi_\lambda,\delta \rangle> 0.        
		\end{displaymath}		
	\end{proposition}
    \begin{proof}
        Let us look at the following diagram which almost explains the proof. The nodes of the first diagram consist of subgroups of $ S_n$ and those of the second diagram consist of irreducible representations of the corresponding subgroups in the first diagram.

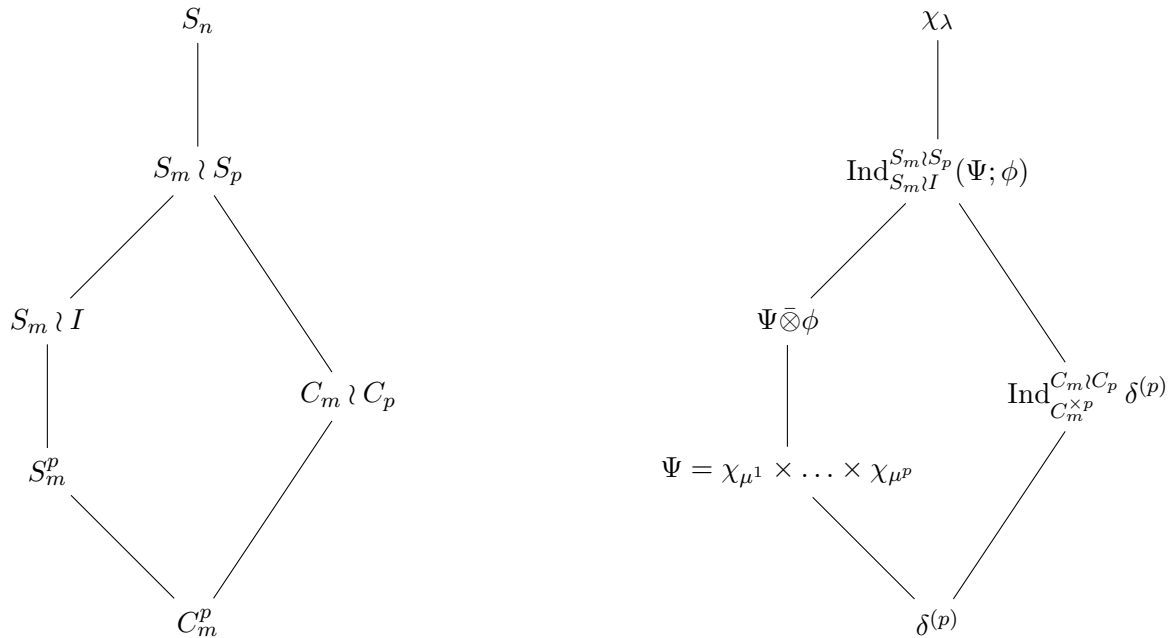
\begin{center}
\begin{tikzpicture}
    \node (Sn) at (0,0) {$S_n$};
    \node (SmSp) at (0,-2) {$S_m\wr S_p$};
    \node (SmI) at (-2,-4) {$S_m\wr I$};
    \node (Smp) at (-2,-6) {$S_m^p$};
    \node (Cmp) at (0,-8) {$C_m^p$};
    \node (CmCp) at (2,-5) {$C_m\wr C_p$};
    
    \draw (Sn) -- (SmSp);
    \draw (SmSp) -- (SmI);
    \draw (SmI) -- (Smp);
    \draw (Smp) -- (Cmp);
    \draw (SmSp) -- (CmCp);
    \draw (CmCp) -- (Cmp);
\end{tikzpicture}
\hspace{3cm}
\begin{tikzpicture}
    \node (Sn) at (0,0) {$\chi_\lambda$};
    \node (SmSp) at (0,-2) {$\Ind_{S_m \wr I}^{S_m\wr S_p} (\Psi;\phi)$};
    \node (SmI) at (-2,-4) {$\Psi\bar{\otimes}\phi$};
    \node (Smp) at (-2,-6) {$\Psi=\chi_{\mu^1}\times \dotsc\times \chi_{\mu^p}$};
    \node (Cmp) at (0,-8) {$\delta^{(p)}$};
    \node (CmCp) at (2,-5) {$\Ind_{C_m^{\times p}}^{C_m\wr C_p} \delta^{(p)}$};
    
    \draw (Sn) -- (SmSp);
    \draw (SmSp) -- (SmI);
    \draw (SmI) -- (Smp);
    \draw (Smp) -- (Cmp);
    \draw (SmSp) -- (CmCp);
    \draw (CmCp) -- (Cmp);
\end{tikzpicture}
\end{center}

    Let us discuss the nodes of the above diagrams more carefully now.
    Let $\Psi=\chi_{\mu^1}\times \dotsc \times \chi_{\mu^p}$.
    From \cite[4.1.25]{JamesKerber}, we have that the normalizer of the subgroup $S_m^{\times p}$ in $S_n$ is equal to the image of $S_m\wr S_p$ under the isomorphism $\Gamma$. 
    Since $\langle\Res^{S_n}_{S_m^{\times p}}\chi_\lambda,\Psi\rangle\neq 0$, there exists an irreducible representation $\theta\in Irr(S_m \wr S_p|\Psi)$ such that $\chi_\lambda\in Irr(S_n,\theta)$. From the description of the irreducible representations of $S_m\wr S_p$, we have $\theta=\Ind_{S_m\wr I}^{S_m\wr S_p} \Psi\bar{\otimes} \chi$. Here, $I$ is the Young subgroup of $S_p$ which stablises 
    $\Psi=\chi_{\mu^1}\times \dotsc \chi_{\mu^p}$. Since $\mu^i\neq \mu^j$ for some $i\neq j$, we have that the Young subgroup $I$ does not contain any $p$-cycles. 
    We now show the series of sub-representations. Firstly,
    \begin{align}\label{eq:s_n_to_wreath}
        \Res_{S_m\wr S_p}^{S_n} \chi_\lambda &\geq   \Ind_{S_m\wr I}^{S_m\wr S_p} \Psi\bar{\otimes} \chi.
    \end{align}
    Using Mackey's restriction formula, we have
    \begin{align*}
        \Res_{C_m\wr C_p}^{S_m\wr S_p} \Ind_{S_m\wr I}^{S_m\wr S_p} \Psi\bar{\otimes} \chi=\sum\limits_{g\in S_m\wr I \setminus S_m\wr S_p / C_m\wr C_p}  \Ind^{C_m\wr C_p}_{{(S_m\wr I)}^g\cap C_m\wr C_p} \Res^{{(S_m\wr I)}^g}_{{(S_m\wr I)}^g\cap C_m\wr C_p} ({\Psi\bar{\otimes} \chi})^g,
    \end{align*}
    where for a subgroup $H$, $H^g=gHg^{-1}$ and for a character $\eta$ of $H$, $\eta^g$ is a character of $gHg^{-1}$ defined by $\eta^g(x)=\eta(g^{-1}xg)$.
    Considering the summand correspond to the double coset representative $g=e$, we obtain
    \begin{align*}
        \Res_{C_m\wr C_p}^{S_m\wr S_p} \Ind_{S_m\wr I}^{S_m\wr S_p} \Psi\bar{\otimes} \chi &\geq \Ind^{C_m\wr C_p}_{{(S_m\wr I)}\cap C_m\wr C_p} \Res^{{(S_m\wr I)}}_{{(S_m\wr I)}\cap C_m\wr C_p} {\Psi\bar{\otimes} \chi} \\
        &= \Ind^{C_m\wr C_p}_{C_m^{\times p}} \Res^{(S_m\wr I)}_{C_m^{\times p}} {\Psi\bar{\otimes} \chi} \\
        &= \Ind^{C_m\wr C_p}_{C_m^{\times p}} \Res^{S_m^{\times p}}_{C_m^{\times p}} \Res^{(S_m\wr I)}_{S_m^{\times p}} {\Psi\bar{\otimes} \chi} \\
        &\geq \Ind^{C_m\wr C_p}_{C_m^{\times p}} \Res^{S_m^{\times p}}_{C_m^{\times p}} \Psi
    \end{align*}
    \begin{equation}\label{eq:SmSp_to_CmCp}
    \Res_{C_m\wr C_p}^{S_m\wr S_p} \Ind_{S_m\wr I}^{S_m\wr S_p} \Psi\bar{\otimes} \chi \geq \Ind^{C_m\wr C_p}_{C_m^{\times p}} \tilde{\delta^{(p)}},
    \end{equation}
    where the second equality follows from the fact that the Young subgroup $I$ does not contain any $p$-cycles and the remaining equalities are from the hypotheses of the statement of the theorem.
    Using Equations~\eqref{eq:s_n_to_wreath},\eqref{eq:SmSp_to_CmCp}, we obtain 
    \begin{align*}
    \Res_{C_m\wr C_p}^{S_n} \chi_\lambda\geq \Ind^{C_m\wr C_p}_{C_m^{\times p}} \tilde{\delta^{(p)}},
    \end{align*}
    which proves the former statement of the theorem.

    From Lemma~\ref{lemma:mn-cycle}, we get that $\eta(\sigma_{mp})\in C_m\wr C_p$ is an $mn$-cycle.
    Since $C_n$ and the cyclic subgroup $D$ generated by $\eta(\sigma_{mn})$ are conjugate subgroups of $S_n$, $\Ind_{C_n}^{S_n} \delta$ and $\Ind_{D}^{S_n} \delta$ are isomorphic $S_n$ representations, where $\delta$ is a linear character of $D$ obtained by sending $\eta(\sigma_{mn})$ to $\delta(w_n)$.
    To prove the latter statement of the theorem, it suffices to show that $\langle\Res_{D}^{ S_m}\chi_\lambda,\delta \rangle> 0$.
    By abuse of notation, let us denote $D$ by $C_n$.
    Using Mackey's restriction formula, we get
    \begin{align*}
      \Res^{C_m\wr C_p}_{C_n}\Ind^{C_m\wr C_p}_{C_m^{\times p}} \tilde{\delta^{(p)}}  &=\sum\limits_{g\in C_m^{\times p} \setminus C_m\wr C_p / C_n}  \Ind^{C_n}_{{(C_m^{\times p})}^g\cap C_n} \Res^{{C_m^{\times p} }^g}_{{(C_m^{\times p})}^g\cap C_n} {\tilde{\delta^{(p)}}}^g\\
      & \geq  \Ind^{C_n}_{{C_m^{\times p}}\cap C_n} \Res^{{C_m^{\times p} }}_{{C_m^{\times p}}\cap C_n} {\tilde{\delta^{(p)}}}\\
      & =  \Ind^{C_n}_{C_n^p} \Res^{{C_m^{\times p} }}_{C_n^p} {\tilde{\delta^{(p)}}}\\
      & =  \Ind^{C_n}_{C_n^p} {\delta^{p}}
    \end{align*}
    \begin{equation}\label{CmCp_to_Cn}
        \Res^{C_m\wr C_p}_{C_n}\Ind^{C_m\wr C_p}_{C_m^{\times p}} \tilde{\delta^{(p)}} \geq \delta
    \end{equation}
    where $C_n^p$ denotes the subgroup generated by $w_n^p$.
    In the above equations, the second inequality follows by considering the summand corresponds to $g=e$ and the other inequalities are obvious.
    Now the latter statement follows from equations~\eqref{eq:SmSp_to_CmCp},\eqref{CmCp_to_Cn}.
The result is proved.
\end{proof}
    \begin{corollary}
		\label{corollary:Giannelli_three_distinct}
		Let $\delta$ be an irreducible character of the cyclic subgroup $C_n$ of $ S_n$ and let $n=4m$ with $m>1$. Let $\mu^1, \mu^2, \mu^3, \mu^4$ be partitions of $m$, with at least three of them pairwise distinct, such that for all $2 \leq i \leq 4$, $\langle\Res_{C_m}^{ S_m} \chi_{\mu^i},\mathds{1}_{C_m}\rangle\neq0$ and $\langle \Res_{C_m}^{ S_m}\chi_{\mu^1},\delta^{(4)}\rangle_{C_m}\neq0$.
		Let  $\lambda\vdash n$  be such that $\chi_{\mu^1} \times\dots \times \chi_{\mu_4}$ is an irreducible constituent of $\Res_{ S_m^{\times 4}}^{ S_n}\chi_\lambda$. Then  $\Ind_{C_m^{\times 4}}^{C_m\wr C_4} \Tilde{\delta^{(4)}}$ is a sub-representation of  $\Res_{C_m\wr C_4}^{ S_n}\chi_\lambda,$ where  $\Tilde{\delta^{(4)}}:=\delta^{(4)}\times \mathds{1}_{C_m}\times \dotsc \times \mathds{1}_{C_m}$.
		As a consequence, we have
		\begin{displaymath}
			\langle\Res_{C_n}^{ S_n}\chi_\lambda,\delta \rangle> 0.        
		\end{displaymath}		
    \end{corollary}
    \begin{proof}
        The corollary follows from the proof of the Proposition~\ref{proposition:Giannelli} and the following observation. The Young subgroup of $S_4$ which stablises $\{\mu^1,\mu^2,\mu^3,\mu^4\}$ intersects $C_4$ trivially since at least of three elements of $\{\mu^1,\mu^2,\mu^3,\mu^4\}$ are pairwise distinct.
    \end{proof}

  \begin{corollary}
		\label{corollary:Giannelli_rectangular_two_row}
		Let $\delta$ be an irreducible character of the cyclic subgroup $C_n$ of $ S_n$ and let $n=2p$ where $p>10$ is a prime.
        If $\delta^{(p)}$ is the sign character $\chi_{(1,1)}$ of $S_2$, then let $\mu^1=(1,1)$ and 
        $\mu^2=\mu^3=\dotsc=\mu^p=(2)$.
        Otherwise, $\delta^{(p)}$ is the trivial character $\chi_{(2)}$ of $S_2$ and let $\mu^1=\mu^2=(1,1)$, $\mu^3=\mu^4=\dotsc=\mu^p=(2)$.
        Then $\chi_{\mu^1} \times\dots \times \chi_{\mu^p}$ is an irreducible constituent of $\Res_{ S_2^{\times p}}^{ S_n}\chi_\lambda$ where $\lambda=(p,p)$.
        Then $\Ind_{S_2^{\times p}}^{S_2\wr C_p} \chi_{\mu^1}\times\dotsc\times\chi_{\mu^p}$ is a sub-representation of  $\Res_{S_2\wr C_p}^{S_n}\chi_\lambda.$
		As a consequence, we have
		\begin{displaymath}
			\langle\Res_{C_n}^{ S_n}\chi_\lambda,\delta \rangle> 0.        
		\end{displaymath}		
	\end{corollary}
    \begin{proof}
        The proof follows from the proof of the Proposition~\ref{proposition:Giannelli}.
    \end{proof}

    We have a couple of interesting lemmas.
    
	\begin{lemma}\label{lemma:alphabeta}
		For every partition $\lambda$ of $p+q$, and every partition $\alpha$ of $p$ that is contained in $\lambda$, there exists a partition $\beta$ of $q$ such that $s_\alpha s_\beta\geq s_\lambda$.
		\begin{proof}
        We reproduce the proof of \cite[Lemma 2.1]{Inv_vectors}, as it will be used extensively in this article.
			Let $T_{\lambda\alpha}$ denote the skew-tableau obtained by putting $i$ in the $i$th cell (from the top) of each column of $\lambda/\alpha$.
			Let $\beta$ be the weight of $T_{\lambda\alpha}$.
            Clearly, $T_{\lambda/\alpha}$ is semi standard.
			Since every $i+1$ occurs below an $i$ in the same column, the reverse row reading word is a lattice permutation.
			The Littlewood-Richardson rule implies that $s_\alpha s_\beta\geq s_\lambda$.
			For example, if $\lambda = (5,4,4,1)$ and $\alpha = (3,2,1)$ then
			\ytableausetup{boxsize=1.2em}
			\begin{displaymath}
				T_{\lambda\alpha} = \begin{ytableau}
					*(yellow)  &*(yellow) &*(yellow) & 1 & 1 \\
					*(yellow)&*(yellow)  & 1 & 2\\
					*(yellow) & 1 & 2 & 3\\
					1 \\
				\end{ytableau},
			\end{displaymath}
			and $\beta$ is $(5,2,1)$.
		\end{proof}
	\end{lemma}

 	\begin{lemma}\label{lemma:choose-beta}
		Given integers $p\geq 2$, $q\geq 1$, and a partition $\lambda\vdash (p+q)$ different from $(1^{(p+q)})$, there exists a partition $\beta\vdash q$ such that $f_{(q)}\geq s_\beta$ and $\beta\subset\lambda$.
	\end{lemma}
	\begin{proof}
	    See \cite[Lemma 2.3]{Inv_vectors}.
	\end{proof}
	
\section{Proof of Thorem~\ref{theorem:Main_min}} \label{section:Proof of Thorem theorem:Main_min}

	\begin{lemma}
		\label{lemma:base}
		If $p>5$ is an odd positive integer, then 
		\begin{equation}\label{eq:trivial_trivial_odd}
			\sum\limits_{\substack{\alpha,\beta\vdash p \\ \alpha\neq\beta\\ \alpha\neq(p-1,1),(2,1^{p-2})\\ \beta\neq (p-1,1),(2,1^{p-2}) }}s_\alpha s_\beta\geq s_\lambda
		\end{equation}
		for all partitions $\lambda$ of $2p$ except when 
		$\lambda=(2p),(2p-1,1),(p,p),(2^p),(2,1^{2p-2}),(1^{2p}).$

        If $p> 5$ is an odd integer, then 
		\begin{equation}\label{eq:non-trivial_trivial_odd}
			\sum\limits_{\substack{\alpha,\beta\vdash p \\ \alpha\neq\beta\\ \alpha\neq (p),(1^p)\\ \beta \neq (p-1,1),(2,1^{p-2})  }}s_\alpha s_\beta\geq s_\lambda
		\end{equation}
		for all partitions $\lambda$ of $2p$ except when 
		$\lambda=(2p),(p,p),(2^p),(1^{2p}).$

		If $p>5$ is an even positive integer, then 
		\begin{equation}\label{eq:trivial_trivial_even}
			\sum\limits_{\substack{\alpha,\beta\vdash p \\ \alpha\neq\beta\\ \alpha\neq(p-1,1),(1^{p})\\ \beta\neq (p-1,1),(1^{p}) }}s_\alpha s_\beta\geq s_\lambda
		\end{equation}
		for all partitions $\lambda$ of $2p$ except when      
		
		$\lambda=(2p),(2p-1,1),(p,p),(3,1^{2p-3}),(2^p),(2,2,1^{2p-4}),(2,1^{2p-2}),(1^{2p}).$
		
		If $p>5$ is an even positive integer, then 
		\begin{equation}\label{eq:non-trivial_trivial_even}
			\sum\limits_{\substack{\alpha,\beta\vdash p \\ \alpha\neq\beta\\ \alpha\neq (p),(1^p) \\ \beta\neq (p-1,1),(1^p) }}s_\alpha s_\beta\geq s_\lambda
		\end{equation}
		for all partitions $\lambda$ of $2p$ except when 
		$\lambda=(2p),(p,p),(3,1^{2p-3}),(2^p),(2,2,1^{2p-4}),(2,1^{2p-2}),(1^{2p}).$

	\end{lemma}
	\begin{proof}
Let us fix some notations.
For a positive integer $p$, let
\begin{gather*}
        \label{eq:chiplusminus}
        A_p = 
        \begin{cases}
             \{(p-1,1),(2,1^{p-2})\}& \text{if $p$ is odd,}\\
            \{(p-1,1),(1^{p})\} & \text{if $p$ is even.}
        \end{cases}
        \\
        B_p =\{(p),(1^p)\}. 
\end{gather*}

We have two equations in the statement of the theorem to consider depending on $p$ is odd or even.
If $p$ is odd, then for Equation~\eqref{eq:trivial_trivial_odd} (resp. \eqref{eq:non-trivial_trivial_odd}), we must produce partitions $\alpha,\beta$ of $p$ such that $\alpha\neq\beta$, $\alpha,\beta\notin A_p$ (resp. $\alpha\notin B_p$ and $\beta\notin A_p$). Similarly, for $p$ even.

        
	The proof uses a case-by-case analysis.

        Let $\lambda$ be a partition of $2p\geq 12$. The theorem is easy to verify when $\lambda=(2p)$ or $(1^{2p})$.
        Thus we may assume that $\lambda$ is equal to neither $(2p) \text{ nor } (1^{2p})$.
		\begin{itemize}
			\item $\lambda\supset (p-1,1)$ 
			\begin{itemize}
				\item Suppose that $\lambda_2>1$.
				Then choose $\alpha=(p-2,2)$, and choose $\beta$ as in the proof of  Lemma~\ref{lemma:alphabeta}. Since $\alpha\notin A_p\cup B_p$, if $\beta\notin A_p$ and $\beta\neq\alpha$, then we are done. Otherwise, one of the following situations occurs.
				\begin{itemize}
					%
					
					\item $\beta = (p-1,1)$.
					Since each column of $T_{\lambda\alpha}$ in the proof of Lemma~\ref{lemma:alphabeta} is filled with integers $1,2,\cdots$ in increasing order, $\lambda/\alpha$ has $p-2$ columns with one cell and one column with two cells.
					Either the first column, the third column, or the $(p-1)$st column can have two cells.
					
					If the first column of $\lambda/\alpha$ has two cells, then we may change $\beta$ to $(p-2,1,1)$ by constructing a skew-tableau as shown in the example below.
					\begin{displaymath}
						\ydiagram{5,2}*[*(yellow) 2]{0,0,1}*[*(yellow) 3]{0,0,0,1}*[*(lime) 1]{5+2,2+3,1}
					\end{displaymath}
					Note that the reverse reading word will be a lattice permutation because $\lambda\supset (p-1,1)$, so there will be at least one cell in the first row of $\lambda/\alpha$.
					
					If the third column of $\lambda/\alpha$ has two cells, then we may change $\beta$ to $(p-3,3)$ as shown in the example below.
					\begin{displaymath}
						\ydiagram{5,2}*[*(yellow) 1]{0,2+1}*[*(yellow) 2]{0,0,2+1}*[*(lime) 1]{5+2,2+3}*[*(lime) 2]{0,0,3}
					\end{displaymath}
					Note that the reverse reading word will be a lattice permutation because at least three $1$ occurs above the third (since $p\geq 6$).
					
					If the $(p-1)$st column of $\lambda/\alpha$ has two cells, we may change $\alpha$ to $(p-3,3)$ (note that $p\geq 6$, so $\alpha$ is a partition).
					We may choose $\beta=(p-2,2)$ as shown in the example below.
					\begin{displaymath}
						\ydiagram{5,2}*[*(yellow) 1]{5+1}*[*(yellow) 2]{0,5+1}*[*(lime) 1]{6+2,2+3,1}
						\longrightarrow
						\ydiagram{4,3}*[*(yellow) 1]{5+1}*[*(yellow) 2]{0,5+1}*[*(lime) 1]{6+2,3+1,1}*[*(lime) 2]{0,4+1}*[*(lime) 1]{4+1}
					\end{displaymath}
					\item $\beta = (2,1^{p-2})$ with $p$ odd.
					This would mean that $\lambda/\alpha$ has two columns, having $1$ and $p-1$ cells respectively.
					But since $\lambda\supset (p-1,1)$, the $(p-1)$st cell in the first row lies in a column of length one.
					Since $p\geq 6$, the other column of $\lambda/\alpha$ has to be the first one.
					In particular, $\lambda/\alpha$ is a vertical strip, so we can replace $\beta$ with $(1^p)$.
					
					\item $\beta=(1^p)$ with $p$ even.
					In this case, $\lambda/\alpha$ has only one column.
					But $\lambda/\alpha$ contains a cell at the position $(1,p-1)$ and hence all the cells of $\lambda/\alpha$ lie in the $p-1$th column.
					It follows that this case cannot occur.
					\item $\beta=(p-2,2)$.
					In this case, the shape $\lambda/\alpha$ contains $p-4$ columns with a single cell and exactly two columns with two cells.    
					If those two column numbers are $p-1$ and $p$, then we get $\lambda=(p,p)$, for which all the equations in the statement of the lemma are does not hold.
					If those two column numbers are $3$ and $p-1$, then we may change $\beta$ to $(p-3,3)$ as shown in the example below.
					\begin{displaymath}
						\ydiagram{5,2}*[*(yellow) 1]{5+1,2+1}*[*(yellow) 2]{0,5+1,2+1}*[*(lime) 1]{6+1,2+3,0+2}*[*(lime) 2]{0,0,2+1}
                        \longrightarrow
                        \ydiagram{5,2}*[*(yellow) 1]{5+1,2+1}*[*(yellow) 2]{0,5+1,2+1}*[*(lime) 1]{6+1,2+3,1}*[*(lime) 2]{0,0,2+1}*[*(magenta) 2]{0,0,1+1}
					\end{displaymath}
					If those two column numbers are $1$ and $p-1$, then $\lambda=(p-1,p-1,1,1)$. Therefore, we may change $\beta$ to $(p-3,1,1)$ as shown in the example below.
					\begin{displaymath}
						\ydiagram{5,2}*[*(lime) 1]{5+0,2+3}*[*(lime) 2]{0,6+0,0}*[*(yellow) 1]{5+1,2+3,1}*[*(yellow) 2]{0,5+1,1+0,0+1}*[*(yellow) 3]{0,0,0,0}
                        \longrightarrow
                        \ydiagram{5,2}*[*(lime) 1]{5+0,2+3}*[*(lime) 2]{0,6+0,0}*[*(yellow) 1]{5+1,2+3,1}*[*(yellow) 2]{0,5+1,1+0}*[*(yellow) 3]{0,0,0,0+1}
					\end{displaymath}
					If those two columns numbers are $3$ and $4$, then we can choose $\alpha=(p-3,3)$ and $\beta=(p-2,1,1)$ as shown in the example below.
					\begin{displaymath}
							\ydiagram{5,2}*[*(yellow) 1]{0,2+2,0}*[*(yellow) 2]{0,0,2+2}*[*(lime) 1]{5+1,2+0,0+3}*[*(yellow) 2]{0,3+1,0}				
                            \longrightarrow
                            \ydiagram{4,3}*[*(yellow) 2]{0,0,2+0}*[*(lime) 1]{4+2,2+0,0+3}*[*(yellow) 2]{0,3+1,0}*[*(yellow) 3]{0,0,3+1}
					\end{displaymath}
					The first row of $\lambda/\alpha$ contains at least one cell (since $p\geq6$), hence the reverse row reading word is a lattice permutation.
					If those two columns numbers are $1$ and $3$, then we may choose $\beta=(p-2,1,1)$ by incrementing the entry $2$ in the cell $(4,1)$ by $1$.
					Finally, the remaining case is that the column numbers are $1$ and $2$. Now we may choose $\beta=(p-4,2,2)$ by incrementing all the entries in the first two columns by $1$. Again, the reverse row reading word is a lattice permutation since $p\geq 6$.
				\end{itemize}
				\item Suppose $\lambda_2=1$ and $\lambda_1\geq p$.
				Let us replace $\beta$ by $(p)$ and choose $\alpha$ as in the proof of Lemma~\ref{lemma:alphabeta}.
                Since $\beta\notin A_p$, if $\alpha\notin A_p$ (rep. $\alpha\notin B_p$) and $\beta\neq \alpha$, then we are done.
                Otherwise one of the following cases occurs.
				\begin{itemize}
					\item $\alpha = (p-1,1)$.
					In this case, $\lambda/\beta$ has $p-1$ columns, of which one column has exactly two cells (and the others have only one cell).
					The column with two cells has to be the first column of $\lambda/\beta$, since $\lambda$ is a hook.
					Incrementing its entries by $1$ allows us to replace $\beta$ by $(p-2,1,1)$.
					\item $\alpha = (2,1^{p-2})$ with $p$ odd.
					Then $\lambda/\beta$ has two columns, having $1$ and $p-1$ cells respectively.
					Since $\lambda$ is a hook, the column with one cell must lie in the first row and the column with $p-1$ cells has to be the first column.
					Incrementing the entries of the first column by $1$ allows us to replace $\alpha$ by $(1^p)$.
					\item $\alpha=(1^p)$ with $p$ even.
					All the cells of $\lambda/\beta$ must lie in the first column.
					We may replace $\beta$ by $(p-2,1,1)$ and $\alpha$ by $(2,1^{p-2})$ as shown in the following example.
					\begin{displaymath}
						\ytableaushort{{}{}{}{}{}{},1,2,3,4}*[*(lime)]{0,1,1,1,1}\longrightarrow
						\ytableaushort{{}{}{}{}11,{},{},2,3}*[*(lime)]{4+2,0,0,1,1}
					\end{displaymath}
					
					\item $\alpha=(p).$
					Then $\lambda=(p+r,p-r)$ with $0\leq r <p$. 
					For $2\leq r\leq p-2$, we may choose $\alpha=(p-2,2)$
					as shown in the example below.
					
					\begin{displaymath}
						\ydiagram{7}*[*(lime) 2]{0,2+2}*[*(lime) 1]{7+2,0+2}*[*(lime) 1]{0,0,0}
					\end{displaymath}
					
					Otherwise $\lambda$ must be one of $(2p-1,1),(p+1,p-1),(p,p)$. 
					Suppose that $\lambda=(p,p)$. The fact that $c_{\alpha\beta}^{\lambda}>0$ with $\alpha,\beta\vdash p$ implies that $\alpha=\beta=(p-k,k)$ and $c_{\alpha\beta}^{\lambda}=1.$ Hence, when $\lambda=(p,p)$, no equation in the lemma is satisfied.					
					Now if $\lambda=(2p-1,1)$, then we may choose $\alpha=(p-1,1)$ and $\beta=(p)$ for Equations~\eqref{eq:non-trivial_trivial_odd},\eqref{eq:non-trivial_trivial_even}. The other equations are not satisfied when $\lambda=(2p-1,1)$.
					
					If $\lambda=(p+1,p-1)$, then we may choose $\alpha=(p-2,2)$ and $\beta=(p-3,3)$ as shown in the example below.
					
					\begin{displaymath}
						\ydiagram{4,2}*[*(lime) 2]{0,2+3}*[*(lime) 1]{4+3,0}
					\end{displaymath}
					
					
					
				\end{itemize}
				
				\item Suppose $\lambda_2=1$ and $\lambda_1=p-1$.
					In other words, $\lambda=(p-1,1^{p+1})$.
					Replace $\alpha$ by $(p-2,1,1)$.
					We could replace $\beta$ by $(2,1^{p-2})$ or $(1^{2p})$ depending on parity of $p$.
			\end{itemize}
		\end{itemize}
		\subsubsection*{Case 2: Suppose that $\lambda\supset (2,1^{p-2})$.}
		If $p$ is odd, then Equations~\eqref{eq:trivial_trivial_odd},~\eqref{eq:non-trivial_trivial_odd} are invariant under the involution $\omega$ which takes $s_\lambda$ to $s_{\lambda'}$.
		Thus in this case we can conjugate $\lambda$ so that $\lambda'$ contains $(p-1,1)$. Now we can use the previous case.
		Finally, the partitions for which the equations are unsatisfied are conjugate to those partitions for which the equations are unsatisfied in the previous case . Namely, the partitions $(2,1^{2p-2}),(2^p),(1^{2p})$ for \eqref{eq:trivial_trivial_odd} and  $(2^p),(1^{2p})$ for \eqref{eq:non-trivial_trivial_odd}.
		
		Now we shall consider the other equations \eqref{eq:trivial_trivial_even}, \eqref{eq:non-trivial_trivial_even}; thus $p$ is even.
		\begin{itemize}
			\item Suppose $\lambda_2\geq 2$. We may replace $\alpha$ by $(2,2,1^{p-4})$.
			Choose $\beta$ as in the proof of Lemma~\ref{lemma:alphabeta}.
			If $\beta\notin B_p$ and $\beta\neq\alpha$, then we are done.
			Otherwise one of the following cases must occur.
			\begin{itemize}
				\item $\beta=(p-1,1)$.
				In this case, $\lambda/\alpha$ has $p-1$ columns, with one column having two cells.
				The column with two cells has to be one of the first three columns.
				
				If the first or second column of $\lambda/\alpha$ has two cells, increment the entries of those two cells by $1$ to replace $\beta$ by $(p-2,1,1)$.
				Since $p\geq 6$, $\lambda/\alpha$ has at least one cell in the first row, so the reverse row reading word is a lattice permutation.
				
				Suppose the third column of $\lambda/\alpha$ has two cells.
				Since $\lambda\supset (2,1^{p-2})$, the first column of $\lambda/\alpha$ must have exactly one cell.
				Changing the entry of this cell from $1$ to $3$ gives us an LR-tableau of weight $(p-2,1,1)$.
				
				\item $\beta=(1^p)$ with $p$ even.
				In this case, $\lambda$ must be equal to $(2,2,1^{2p-4})$. One can easily check that indeed when $\lambda=(2,2,1^{2p-4})$ Equations \eqref{eq:trivial_trivial_even},~\eqref{eq:non-trivial_trivial_even} are not satisfied.
				\item $\beta=(2,2,1^{p-4})$.
				In this case, $\lambda/\alpha$ must have exactly two columns with one having $2$ cells and the other having $p-2$ cells. Notice that one of these columns must be the first column since $\lambda/\alpha$ contains at least one cell in the first column. If the second column contains $p-2$ cells, then $\lambda=(2^p)$. One can easily see that when $\lambda=(2^p)$, all the equations in the lemma are unsatisfied by using the involution $\omega$ and our argument for $\lambda=(p,p)$. Therefore, the first column must contain $p-2$ cells. Hence $\lambda=(2^4,1^{2p-8})\text{ or } (3,3,1^{2p-6}).$ In both cases the skew shape $\lambda/\alpha$ is a vertical strip, we may replace $\beta$ by $(1^p)$. 
			\end{itemize}
			\item Suppose $\lambda_2=1$ (so that $\lambda$ is a hook) and $\lambda_1\geq 3$.
			Replace $\alpha$ by $(3,1^{p-3})$ and choose $\beta$ as in the proof of
            Lemma~\ref{lemma:alphabeta}.
			Since $\alpha\notin A_p\cup B_p$, if $\beta\notin A_p$ and $\beta\neq\alpha$, then we are done.
			Otherwise one of the following cases must occur.
			\begin{itemize}
				\item $\beta = (p-1,1)$. In this case, $\lambda$ has $p-1$ columns, one having two cells and the other having only one cell.
				Since $\lambda$ is a hook, only the first column of $\lambda/\alpha$ can have two cells.
				Incrementing the entries in these cells by one will allow us to change $\beta$ to $(p-2,1,1)$.
				\item $\beta = (1^p)$.
				In this case, $\lambda$ must be $(3,1^{2p-3}).$ One can easily see that when $\lambda=(3,1^{2p-3})$ Equations \eqref{eq:non-trivial_trivial_even},\eqref{eq:trivial_trivial_even} are unsatisfied. 
				\item $\beta=(3,1^{p-3}).$ 
				In this case $\lambda/\alpha$  has exactly three columns with one having $p-2$ cells and the other two columns having one cells respectively.
				Since $\lambda$ is a hook, thus $\lambda=(5,1^{2p-5})$.
				We may replace $\beta$ by $(1^p)$.
			\end{itemize}
			\item Suppose $\lambda_2=1$ and $\lambda_1=2$.
			In this case $\lambda$ must be $(2,1^{2p-2})$, for which both Equations
            \eqref{eq:trivial_trivial_even},~\eqref{eq:non-trivial_trivial_even} are unsatisfied.
			\end{itemize}
			\noindent
   Finally, let us consider the case when $\lambda$ does not contain both $(p-1,1)$ and $(2,1^{p-2})$.
    Equivalently, $\lambda_1,\lambda_1'<p-1$. Note that we have assumed $\lambda\neq(2p),(1^{2p})$.
			Thus, $l(\lambda)>2$ and $\lambda_1>2$.
			Let $\alpha$ be a partition of $p$ contained in $\lambda$ which is maximal in the dominance order. Clearly, $\alpha\notin A_p\cup B_p$. Let us choose $\beta$ as in the proof of Lemma~\ref{lemma:alphabeta}. If $\beta\neq \alpha$, then we are done since $\beta\notin A$.
   
   Otherwise, $\alpha=\beta$.
	Let us replace $\alpha$ using the following procedure.
			\begin{enumerate}
				\item Choose all the cells in the first row of $\lambda$.
				\item Since $\lambda_1<p-1$, we need to choose more cells to construct a partition $\alpha$ of $p$. In this step, we choose cells in the first column from the top.
				\item Choose the rest of the cells for $\alpha$ in any manner the reader wish if required.
			\end{enumerate}
	Now we have constructed our $\alpha$.		
   Let us choose $\beta$ as in the proof of Lemma~\ref{lemma:alphabeta}.
   
	Suppose that we used step 3 of the algorithm in the construction of $\alpha$. Then clearly $\beta_1<\alpha_1 $ and hence $\alpha\neq\beta$. Thus, we are done.
    
		Otherwise, $\alpha=(\lambda_1,1^{p-\lambda_1})$.
			If $\lambda_3>2$, then clearly $\beta_2>1$. Therefore, $\alpha\neq\beta$ and we are done.
			Hence, we may assume that $\lambda_3\leq 2$.
            			
			Suppose that $\lambda_3=2$. We also have $\lambda_1=\lambda_2$ since $\alpha_1=\beta_1$.
			Note that the first column of $\lambda/\alpha$ contains exactly one cell due to $\alpha_1=\beta_1$, $\alpha_2=1$ and $\lambda_3=2$.
			Since $\alpha=\beta=(\lambda_1,1^{p-\lambda_1})$, we must have $\lambda=(\lambda_1,\lambda_1,2^{p-\lambda_1})$.
			In this case we may choose $\alpha=(\lambda_1-1,1^{p-\lambda_1+1})$ and $\beta=(\lambda_1-1,2^2,1^{p-\lambda_1+3})$ as shown in the example below. Note that $\alpha,\beta\notin\{(p),(p-1,1),(2,1^{p-2}),(1^p)\}$ due to $\lambda_1,\lambda_1'<p.$
			
			\begin{displaymath}
				\ydiagram{4,1,1}*[*(lime) 1]{0,1+3,0,0+1}*[*(lime) 2]{0,0,1+1,0,0}*[*(lime) 3]{0,0,0,1+1}\longrightarrow
				\ydiagram{3,1,1,1}*[*(lime) 1]{3+1,1+2,0,0}*[*(lime) 2]{0,3+1,1+1,0,0}*[*(lime) 3]{0,0,0,1+1}
			\end{displaymath}
			
			Finally let $\lambda_3=1$ and hence $\lambda=(\lambda_1,\lambda_1,1^{2p-2\lambda_1}).$
			We may choose $\alpha=(\lambda_1-1,1^{p-\lambda_1+1})$ and $\beta=(\lambda_1,2,1^{p-\lambda_1-2})$ as shown in the example below.

			\begin{displaymath}
				\ydiagram{4,1,1}*[*(lime) 1]{0,1+3,0,0+1}*[*(lime) 2]{0,0,0,0,0+1}*[*(lime) 3]{0,0,0,0,0,0+1}\longrightarrow
				\ydiagram{3,1,1,1}*[*(lime) 1]{3+1,1+2,0,0,0+1}*[*(lime) 2]{0,3+1,0,0,0,0+1}
			\end{displaymath}

			
			
			
			
			
			
			This completes the proof.
		\end{proof}
		
		\begin{lemma}
			\label{lemma:more than two parts}
			Let $p\geq 7$ be a positive odd integer.
			Then 
			\begin{equation}\label{eq:trivial_trivial_trivial_odd}
				\sum\limits_{\substack{\alpha\vdash p \\ \alpha\neq(2p),(2p-1,1),(p,p)\\ \alpha\neq (2^p),(2,1^{2p-2}),(1^{2p})}}s_\alpha \sum\limits_{\substack{\beta\vdash p \\ \beta\neq (p-1,1),(2,1^{p-2})}}s_\beta \geq s_\lambda
        \end{equation}
			for all partitions $\lambda$ of $3p$ except when 
			$\lambda=(3p),(3p-1,1),(2,1^{3p-2}),(1^{3p}),$ and 
			
			\begin{equation}\label{eq:non_trivial_trivial_odd}
				\sum\limits_{\substack{\alpha\vdash p \\ \alpha\neq(2p),(p,p)\\ \alpha\neq (2^p),(1^{2p})}}s_\alpha \sum\limits_{\substack{\beta\vdash p \\ \beta\neq (p-1,1),(2,1^{p-2})}}s_\beta \geq s_\lambda
			\end{equation}
			for all partitions $\lambda$ of $3p$ except when 
			$\lambda=(3p),(1^{3p}).$ 
			
		\end{lemma}
		\begin{proof}
			We shall prove the first statement of the lemma.
			The latter statement follows from the former statement and the fact that \eqref{eq:non_trivial_trivial_odd} holds when $\lambda=(3p-1,1),(2,1^{3p-2}).$
			
			Let $H:=\{(2p),(2p-1,1),(p,p), (2^p),(2,1^{2p-2}),(1^{2p})\}.$
			For $\lambda=(3p) \text{ or } (1^{3p})$, the lemma is obvious.
			We may assume that $\lambda\vdash 3p$ and $\lambda\neq (3p),(1^{3p}).$
			We may choose $\alpha\vdash 2p$, $\beta\vdash p$ such that $s_\alpha s_\beta \geq s_\lambda$ with $\beta\neq (p-1,1),(2,1^{p-2})$ by Lemma~\ref{lemma:choose-beta}.
			If $\alpha\notin H$, then we are done.
			Otherwise one of the following cases must occur.

			\begin{itemize}
				\item $\alpha=(2p)\text{ or } (2p-1,1) $.
				Suppose that $\lambda_1'> 2$.
				Then we may replace $\alpha$ by $(2p-2,1,1)$ and choose $\beta$ as in the proof of Lemma~\ref{lemma:alphabeta}.
				If $\beta\neq (p-1,1),(2,1^{p-2}), $ then we are done.
				Otherwise, one of the following cases must occur.
				\begin{itemize}
					\item $\beta=(p-1,1).$
					Then $\lambda/\alpha$ has exactly $p-1$ columns with one of them having two cells.
					Either the first column or the second column can have two cells.
					If the first column has two cells, then we may increment the entries in the first column by $1$ and we replace $\beta$ by $(p-2,1,1)$.
					If the second column has two cells, then we can increment the last entry $1$ in the second row. Thus we can replace $\beta$ by $(p-2,1,1)$ as shown in the example below.
					
					\begin{displaymath}
						\ydiagram{12,1,1}*[*(yellow) 1]{12+2,2+2,0,0+1}*[*(yellow) 2]{0,0,0}*[*(yellow) 3]{0,0,0,0}*[*(lime) 1]{0,1+1,0}*[*(lime) 2]{0,0,1+1}\longrightarrow
						\ydiagram{12,1,1}*[*(yellow) 1]{12+2,1+2,0,0+1}*[*(yellow) 2]{0,0,0}*[*(yellow) 3]{0,0,0,0}*[*(lime) 1]{0,0,0}*[*(lime) 2]{0,3+1,0}*[*(lime) 3]{0,0,1+1}
					\end{displaymath}
					
					\item $\beta=(2,1^{p-2})$.
					In this case $\lambda/\alpha$ has exactly two columns with one having $p-2$ cells and the other having a single cell.
					Since $\lambda/\alpha$ contains at least one cell in its first row,
					we conclude that $\lambda=(2p,1^p)$. We shall replace $\beta$ by $(1^p)$.
				\end{itemize}
				Suppose that $\lambda_1'=2$, i.e., $\lambda$ is a two-row partition $(2p+r,p-r)$ with $1\leq r< p$.
				If $r=p-1$, then $\lambda=(3p-1,1)$. In this case we can easily see that equation~\eqref{eq:trivial_trivial_trivial_odd} does not hold.
				For $1\leq r < p-1$ we may replace $\alpha$ by $(2p-2,2)$ and $\beta$ by $(p)$.

				\item $\alpha=(p,p)$.
				Suppose that $\lambda_1'\geq  3$.
				We replace $\alpha$ by $(p,p-1,1)$ and choose $\beta$ as in the proof of Lemma~\ref{lemma:alphabeta}.
				If $\beta\neq (p-1,1),(2,1^{p-2})$, then there is nothing to prove.
				Otherwise, one of the following cases occurs.
				\begin{itemize}
					\item $\beta=(p-1,1)$.
					In this case, $\lambda/\alpha$ has $p-2$ columns with $p-1$ columns having a single cell and one column having two cells. Then one of the first, second, $p$th or $p+1$th column has two cells.
					If the first column has two cells, then we may replace $\beta$ by $(p-2,1,1)$ as shown in the example below.
					
					\begin{displaymath}
						\ydiagram{7,6,1}*[*(yellow) 1]{7+2,6+1,1+2,0}*[*(yellow) 2]{0,0,0}*[*(yellow) 3]{0,0,0,0}*[*(lime) 1]{0,0,0}*[*(lime) 1]{0,0,0,0+1}*[*(lime) 2]{0,0,0,0,0+1}\longrightarrow
						\ydiagram{7,6,1}*[*(yellow) 1]{7+2,6+1,1+2,0}*[*(yellow) 2]{0,0,0}*[*(yellow) 3]{0,0,0,0}*[*(lime) 1]{0,0,0}*[*(lime) 1]{0,0,0,0}*[*(lime) 2]{0,0,0,0+1}*[*(lime) 3]{0,0,0,0,0+1}
					\end{displaymath}
					
					If the second column has two cells, then we may replace $\beta$ by $(p-2,2)$ as shown in the example below.
					
					\begin{displaymath}
						\ydiagram{7,6,1}*[*(yellow) 1]{7+1,6+1,2+2,0+1}*[*(yellow) 2]{0,0,0}*[*(yellow) 3]{0,0,0,0}*[*(lime) 1]{0,0,1+1,0}*[*(lime) 2]{0,0,0,1+1}\longrightarrow
						\ydiagram{7,6,1}*[*(yellow) 1]{7+1,6+1,2+2,0}*[*(magenta) 2]{0,0,0,0+1}*[*(yellow) 3]{0,0,0,0}*[*(lime) 1]{0,0,1+1,0}*[*(lime) 2]{0,0,0,1+1}
					\end{displaymath}
					If the $p$th column has two cells, then we may replace $\alpha$ by $(p-1,p-1,2)$  and $\beta$ by $(p-2,1,1)$ as shown in the example below.
					
					\begin{displaymath}
						\ydiagram{7,6,1}*[*(yellow) 1]{0,6+0,1+5,0+0}*[*(yellow) 2]{0,0,0}*[*(yellow) 3]{0,0,0,0}*[*(lime) 1]{0,6+1,1+0,0}*[*(lime) 2]{0,0,6+1}\longrightarrow
						\ydiagram{6,6,2}*[*(yellow) 1]{0,6+0,2+4,0+0}*[*(yellow) 2]{0,0,0}*[*(yellow) 3]{0,0,0,0}*[*(lime) 1]{6+1,1+0,0}*[*(lime) 2]{0,6+1}*[*(lime) 3]{0,0,6+1}
					\end{displaymath}
					
					If the $p+1$th column has two cells, then we may replace $\beta$ by $(p-2,2)$ if $(4,1)$th cell lies in $\lambda/\alpha$ as shown in the example below.
					
					\begin{displaymath}
						\ydiagram{7,6,1}*[*(yellow) 1]{8+2,6+1,1+2,0+1}*[*(yellow) 2]{0,0,0}*[*(yellow) 3]{0,0,0,0}*[*(lime) 1]{7+1,1+0,0}*[*(lime) 2]{0,7+1}\longrightarrow
						\ydiagram{7,6,1}*[*(yellow) 1]{8+2,6+1,1+2,0}*[*(yellow) 2]{0,0,0,0}*[*(yellow) 3]{0,0,0,0}*[*(lime) 1]{7+1,1+0,0}*[*(lime) 2]{0,7+1}*[*(magenta) 2]{0,0,0,0+1}
					\end{displaymath}
					
					Otherwise, we may replace $\alpha$ by $(p+1,p-1)$ and $\beta$ by $(p)$ as shown in the example below.
					\begin{displaymath}
						\ydiagram{7,6,1}*[*(yellow) 1]{8+2,6+1,1+2,0}*[*(yellow) 2]{0,0,0}*[*(yellow) 3]{0,0,0,0}*[*(lime) 1]{7+1,1+0,0}*[*(lime) 2]{0,7+1}\longrightarrow
						\ydiagram{8,6,1}*[*(yellow) 1]{8+2,6+2,0+3,0}*[*(yellow) 2]{0,0,0,0}*[*(yellow) 3]{0,0,0,0}*[*(lime) 1]{0,1+0,0}*[*(lime) 2]{0,0}
					\end{displaymath}
					
					\item $\beta=(2,1^{p-2})$.
					In this case, $\lambda/\alpha$ has exactly two columns, with one having $p-1$ cells and the other having a single cell.
					Since $\lambda/\alpha$ has a cell in its second row, thus $\lambda=(p,p,1^p)$.
					We may replace $\beta=(1^p)$ since $\lambda/\alpha$ is a vertical strip.
				\end{itemize}
				
				Note that the LHS of Equation~\eqref{eq:trivial_trivial_trivial_odd} is invariant under the involution $\omega$. Thus, for the other cases $\alpha=(2^p),(2,1^{2p-2}),(1^{2p})$ we use the above cases and the involution $\omega$ to prove the lemma.        
			\end{itemize}
		\end{proof}
		
		\begin{lemma}\label{Lemma:gen_case}
			Let $p\geq 7$ be a positive odd integer.
			Then 
			\begin{equation}\label{eq:trivial_trivial_trivial_trivial_odd}
				\bigg( \sum\limits_{\substack{\beta\vdash p \\ \beta\neq (p-1,1),(2,1^{p-2})}}s_\beta \bigg) ^j\sum\limits_{\substack{\alpha\vdash p \\ \alpha\neq(2p),(2p-1,1),(p,p)\\ \alpha\neq (2^p),(2,1^{2p-2}),(1^{2p})}}s_\alpha  \geq s_\lambda
			\end{equation}
			for all $j\geq 1$ and for all partitions $\lambda$ of $n$ $(=(j+2)p)$ except possibly when 
			$\lambda=(n),(n-1,1),(2,1^{n-2}),(1^{n})$ and 
			
			\begin{equation}\label{eq:non_trivial_trivial_trivial_odd}
				\bigg( \sum\limits_{\substack{\beta\vdash p \\ \beta\neq (p-1,1),(2,1^{p-2})}}s_\beta \bigg) ^j \sum\limits_{\substack{\alpha\vdash p \\ \alpha\neq(2p),(p,p)\\ \alpha\neq (2^p),(1^{2p})}}s_\alpha  \geq s_\lambda
			\end{equation}
			for all $i\geq 1$ and for all partitions $\lambda$ of $n$ $(=(j+2)p)$ except when 
			$\lambda=(n),(1^{n}).$     
		\end{lemma}
		\begin{proof}
			The proof follows using the proof of Lemma~\ref{lemma:more than two parts} iteratively.
		\end{proof}

		\begin{proof}[Proof of Theorem~\ref{theorem:Main_min}]\label{proof_of_Theorem:main_min}
            We prove the theorem by induction on $n$.
            We can easily verify the theorem when $n\leq 26$ by using sage for example.
            So assume that $n>26$.
			Let us also assume that the theorem holds for $k<n$ by induction.
			We prove the theorem when $k=n$.
			
			If $n$ is a prime, then a small calculation yields the theorem.
			One of the simple calculation would be the following.
			Consider the matrix $\rho_\lambda(w_{(n)})$.
			When $\lambda$ is not a hook partition, we have that $\chi_\lambda(w_{(n)})=\text{trace}(\rho_\lambda(w_{(n)}))=0$ by the Murnaghan-Nakayama rule.
			We recall that the characteristic polynomial of $\rho_\lambda(w_{(n)})$ is in $\mathbb{Z}[x]$.
			Since the sum of all the eigenvalues of $\rho_\lambda(w_{(n)})$ is $0$ and the only possible eigenvalues are $p$th roots of unity, we have that all $pth$ roots of unity occurs as an eigenvalue of $\rho_\lambda(w_{(n)})$.
			When $\lambda$ is a hook partition ($\neq (n),(n-1,1),(2,1^{n-2}),(1^n)$), we notice that $\chi_\lambda(w_{(n)})=\text{trace}(\rho_\lambda(w_{(n)}))=(-1)^{n-\lambda_1}.$
            Since $\dim(V_\lambda)>n-1$ and sum of non-trivial $p$th roots of unity is $-1$, we conclude that all $p$th roots of unity occurs as an eigenvalue of $\rho_\lambda(w_{(n)})$.
			When $\lambda$ equal to one of $(n),(n-1,1),(2,1^{n-2}),$ or $(1^n)$, one can easily verify the theorem.
            
            \smallskip
			\noindent Now we may assume that $n$ is a composite number.
			
			Suppose that $n$ is odd. Write $n=mp$ where $p$ is the smallest prime divisor of $n$. Since $n\geq 26$, we have that $m$ is an odd positive integer greater than or equal to $7$.
            Let $\lambda\vdash n$.
            If $\lambda$ is equal to one of $(n),(n-1,1),(2,1^{n-2}),(1^n)$, then one can easily verify the theorem by direct computation.
            Otherwise, using Lemma~\ref{Lemma:gen_case} and Proposition~\ref{proposition:Giannelli}, the theorem follows.
            
            Suppose that $n$ is even.
			We write $n=mp$ where $m=n/2$ and $p=2$. Note that $m\geq 10$.
            Let $\lambda\vdash n$ and not equal to any of $(n),(n-1,1),(n/2,n/2),(3,1^{n-3}),(2^{n/2}),(2,2,1^{n-4}),(2,1^{n-2}),(1^{n}).$
            Then using Lemma~\ref{lemma:base} and Proposition~\ref{proposition:Giannelli}, the theorem follows.

            Finally, if $\lambda$ is equal to any of $(n),(n-1,1),(n/2,n/2),(3,1^{n-3}),(2^{n/2}),(2,2,1^{n-4}),(2,1^{n-2}),(1^{n})$, one can directly compute (for example the inner product of characters) to prove the theorem except may be when $\lambda=(n/2,n/2)$ or  $(2^{n/2})$.
            
            Suppose that $\lambda=(n/2,n/2)$.
            If $n$ is divisible by $4$, then $n=4m$ (note that $m\geq 6$) and we may choose $\mu^1=(m,m), ~\mu^2=(m+1,m-1), ~\mu^3=(m+2,m-2), ~\mu^4=(m+1,m-1)$ in Corollary~\ref{corollary:Giannelli_three_distinct}, to get the theorem.

            Otherwise, $n=mp$ where $p$ is the least prime dividing $n$ other than $2$.
            If $m\geq 6$, then we may choose $\mu^1=(m-2,2)$ and $\mu^2=\mu^3=\dotsc=\mu^p=(m)$ in the proposition~\ref{proposition:Giannelli} and we are done.
            Otherwise, $m=2$ since $m$ is even and not divisible by $4$.
            Now the theorem follows from Corollary~\ref{corollary:Giannelli_rectangular_two_row}.

            Since $\chi_{(n/2,n/2)}\otimes\chi_{(1^n)}=\chi_{(2^{n/2})}$, and we have shown that $\rho_{(n/2,n/2)}(w_n)$ has minimal polynomial $x^n-1$, it follows that the theorem is true when $\lambda=(2^{n/2})$.

            This completes the proof.
            
			
		\end{proof}

\section{Proof of Theorem~\ref{Theorem:main}} \label{section:Proof of Theorem Theorem:main}

The following are the necessary lemmas to prove Theorem~\ref{Theorem:main}.

        \begin{lemma}\label{lemma:Yang_stareletov}
            Let $\gamma=(p,q)$ be a partition of $n$ with $p\geq 3$ and $q\geq1$.
            Then 
            \begin{equation}\label{eq:Yang_Stareletov}
               \ch\Ind_{C_{(p,q)}}^{ S_n} \widetilde{\zeta\times1} \geq s_\lambda
            \end{equation}
            for all partitions $\lambda\vdash n$ except when $\lambda=(n),(1^n)$, where $\zeta$ is any non-real one dimensional character of $C_p$, $1$ denotes the trivial character of $C_q$ and $\widetilde{\zeta\times1}:=\Res_{C_{(p,q)}}^{C_p\times C_q} \zeta\times 1$.
        \end{lemma}		
        \begin{proof}
            Clearly, 
            \begin{equation}\label{eq:non_trivial_trivial_starletov}
                \Ind_{C_{(p,q)}}^{ S_n} \widetilde{\zeta\times1}=\Ind_{C_p\times C_q}^{ S_n}\Ind_{C_{(p,q)}}^{C_p\times C_q} \widetilde{\zeta\times1}\geq \Ind^{ S_n}_{C_p\times C_q} \zeta\times1,
            \end{equation}
        where the last inequality $\geq$ means that the RHS is a subrepresentation of LHS.
        Using Theorem~\ref{theorem:Main_min} and induction in stages, we obtain
        \begin{equation}\label{eq:non_trivial_trivial_yang}
       \ch\Ind^{ S_n}_{C_p\times C_q} \zeta\times1 =\ch \Ind^{ S_n}_{ S_p\times  S_q} \Ind^{ S_p\times  S_q}_{C_p\times C_q} \zeta\times 1 \geq \sum\limits_{\substack{\alpha\vdash p,\beta\vdash q \\ \alpha\neq (p),(1^{p}) \\ f_{q}\geq s_\beta }}s_\alpha s_\beta\geq s_\lambda.
        \end{equation}

        Let $\lambda\vdash n$ not equal to $(n),(1^n).$
        For almost all partitions $\lambda$, we provide a pair of partitions $\alpha,\beta$ such that $\alpha\neq (p),(1^p)$, $f_q\geq s_\beta$ and $s_\alpha s_\beta \geq s_\lambda$.
        Hence the following equation holds for almost all $\lambda$:
        \begin{equation}\label{equation:n_t_t}
            \sum\limits_{\substack{\alpha\vdash p,\beta\vdash q \\ \alpha\neq (p),(1^{p}) \\ f_{q}\geq s_\beta }}s_\alpha s_\beta\geq s_\lambda.
        \end{equation}
        Using equations \eqref{eq:non_trivial_trivial_starletov},~\eqref{eq:non_trivial_trivial_yang},~\eqref{equation:n_t_t} we obtain that $\ch\Ind_{C_{(p,q)}}^{ S_n} \widetilde{\zeta\times1} \geq s_\lambda$.
        For the remaining partitions $\lambda$, we shall do direct computation, as we will see.

        Suppose that $q\geq 6$, one can easily use the proof of statements correspond to Equations~\eqref{eq:non-trivial_trivial_odd}, \eqref{eq:non-trivial_trivial_even} in Lemma~\ref{lemma:base} to get the results where the partitions which does not satisfy Equations~\eqref{eq:non-trivial_trivial_odd}, \eqref{eq:non-trivial_trivial_even} in Lemma~\ref{lemma:base} can be easily checked in the statement of the lemma separately.
        Suppose that $q\leq 5$ and $p\leq 14$.
        For these finitely many cases, one can directly verify that equation~\eqref{equation:n_t_t} holds for all $\lambda\vdash n$ except when $\lambda=(2,1),(3,3),(2,2,2)$ and possibly when  $\lambda\in B$, where $B=\{(n),(2,1^{n-2}),(1^n)\}$.
        In these cases, we may verify Equation \eqref{eq:Yang_Stareletov} directly. 

        Suppose that $q\leq 5$ and $p\geq 15$.
        Let us choose $\alpha\vdash p$ and $\beta\vdash q$ such that $f_q\geq s_\beta$ and $s_\alpha s_\beta\geq s_\lambda$ using Lemma~\ref{lemma:choose-beta}.
        If $\alpha\neq (p),(1^p)$, then we are done.
        \subsubsection*{Case 1: $\alpha=(p)$.}
        Since $\lambda\neq (n)$, we may replace $\alpha$ by $(p-1,1)$ and choose $\beta$ as in the proof of Lemma~\ref{lemma:alphabeta}.
        If $f_q\geq s_\beta$, then we are done.
        Otherwise one of the following cases must occur.
        \begin{itemize}
            \item Suppose that $\beta=(q-1,1)$.
            Then $\lambda/\alpha$ contains $q-1$ columns, of which precisely one column contains two cells. The column of $\lambda/\alpha$ which contains two cells can be either the first or the second since $p\geq 15$ and $q\leq 5$.
            If the first column contains two cells, then incrementing the entries in the first column of $\lambda/\alpha$ by $1$ allows us to replace $\beta$ by $(q-2,1,1)$ as shown below. Note that $q\geq 3$ in this case due to the existence of at least one cell in the first row of $\lambda/\alpha$.
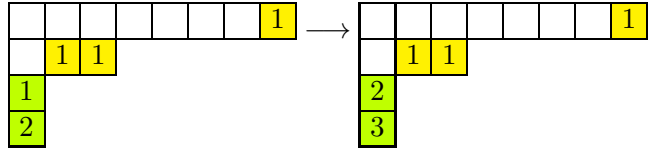
                    \begin{displaymath}
						\ydiagram{7,1}*[*(yellow) 1]{7+1,1+2,0}*[*(yellow) 2]{0,0,0}*[*(yellow) 3]{0,0,0,0}*[*(lime) 1]{0,0,0+1,0}*[*(lime) 2]{0,0,0,0+1}\longrightarrow
						\ydiagram{7,1}*[*(yellow) 1]{7+1,1+2,0}*[*(yellow) 2]{0,0,0}*[*(yellow) 3]{0,0,0,0}*[*(lime) 1]{0,0,0}*[*(lime) 2]{0,0,0+1}*[*(lime) 3]{0,0,0,0+1}
					\end{displaymath}

            If the second column contains two cells, then incrementing the entry in the first column of $\lambda/\alpha$ allows us to replace $\beta$ by $(q-2,2)$ as shown below except when $q=3$.

            \begin{displaymath}
						\ydiagram{7,1}*[*(yellow) 1]{7+1,2+1,0+1}*[*(yellow) 2]{0,0,0}*[*(yellow) 3]{0,0,0,0}*[*(lime) 1]{0,1+1,0}*[*(lime) 2]{0,0,1+1}\longrightarrow
						\ydiagram{7,1}*[*(yellow) 1]{7+1,2+1,0}*[*(magenta) 2]{0,0,0+1}*[*(yellow) 3]{0,0,0,0}*[*(lime) 1]{0,1+1,0}*[*(lime) 2]{0,0,1+1}
			\end{displaymath}

            If $q=3$, then we may replace $\alpha$ by $(p-2,2)$ and $\beta$ by $(3)$.
            
            \item Suppose that $\beta=(2,1^{q-2})$ and $q$ odd.
            Then $\lambda=(p,1^q)$. We may replace $\beta$ by $(1^q)$ since $\lambda/\alpha$ is a vertical strip.
            \item Suppose that $\beta=(1^q)$ and $q$ even.
            Then $\lambda/\alpha$ must be contained in the $p$th column but $q\geq 2$ since $q$ is even. Hence this case cannot occur.
        \end{itemize}
        \subsubsection*{Case 2: $\alpha=(1^p)$.}
        Since $\lambda\neq (1^n)$, we may replace $\alpha$ by $(2,1^{p-2})$ and choose $\beta$ as in the proof of Lemma~\ref{lemma:alphabeta}.
        If $f_q\geq s_\beta$, then we are done.
        Otherwise one of the cases must occur.
        \begin{itemize}
            \item Suppose that $\beta=(q-1,1)$. Then the Young diagram of $\lambda/\alpha$
            contains $q-1$ columns, of which precisely one column contains two cells. The column of $\lambda/\alpha$ which contains two cells can be either the first, second or the third column.            

            If the first column contains two cells, then incrementing the entries in the first column by $1$ allows us to replace $\beta$ by $(q-2,1,1)$ as shown in the example below except when $q\leq 2$.
                  	\begin{displaymath}
						\ydiagram{2,1,1,1,1,1,1,1}*[*(yellow) 1]{2+2,1+1,0,0,0,0,0,0,0}*[*(yellow) 2]{0,0,0,0,0,0,0,0,0}*[*(yellow) 3]{0,0,0,0}*[*(lime) 1]{0,0,0,0,0,0,0,0,0+1,0}*[*(lime) 2]{0,0,0,0,0,0,0,0,0,0+1}\longrightarrow
						\ydiagram{2,1,1,1,1,1,1,1}*[*(yellow) 1]{2+2,1+1,0,0,0,0,0,0,0}*[*(yellow) 2]{0,0,0,0,0,0,0,0,0}*[*(yellow) 3]{0,0,0,0}*[*(lime) 1]{0,0,0,0,0,0,0,0,0,0}*[*(lime) 2]{0,0,0,0,0,0,0,0,0+1}*[*(lime) 3]{0,0,0,0,0,0,0,0,0,0+1}
					\end{displaymath}
            If $q\leq2$, then $\lambda=(2,1^{n-1})$ and in this case one can compute directly to show that Equation~\eqref{eq:Yang_Stareletov} holds.

            If the second column contains two cells or the third column contains two cells, then   incrementing the entry in the first column by $1$ (since the first column of $\lambda/\alpha$ contains precisely one cell), allows us to replace $\beta$ by $(q-2,1,1)$ as shown in the examples below.
            Note that in this case $q\geq 3$.
            \begin{displaymath}
						\ydiagram{2,1,1,1,1,1,1,1}*[*(yellow) 1]{2+2,0,0,0,0,0,0,0,0+1}*[*(yellow) 2]{0,0,0,0,0,0,0,0,0}*[*(yellow) 3]{0,0,0,0}*[*(lime) 1]{0,1+1,0}*[*(lime) 2]{0,0,1+1}\longrightarrow
						\ydiagram{2,1,1,1,1,1,1,1}*[*(yellow) 1]{2+2,0,0,0,0,0,0,0,0}*[*(yellow) 2]{0,0,0,0,0,0,0,0,0}*[*(yellow) 3]{0,0,0,0}*[*(lime) 1]{0,1+1,0}*[*(lime) 2]{0,0,1+1}*[*(magenta) 3]{0,0,0,0,0,0,0,0,0+1} 
                        \hspace{1.5cm}
                        \ydiagram{2,1,1,1,1,1,1,1}*[*(yellow) 1]{3+1,1+1,0,0,0,0,0,0,0+1}*[*(yellow) 2]{0,0,0,0,0,0,0,0,0}*[*(yellow) 3]{0,0,0,0}*[*(lime) 1]{2+1,0}*[*(lime) 2]{0,2+1}\longrightarrow
						\ydiagram{2,1,1,1,1,1,1,1}*[*(yellow) 1]{3+1,1+1,0,0,0,0,0,0,0}*[*(yellow) 2]{0,0,0,0,0,0,0,0,0}*[*(magenta) 3]{0,0,0,0,0,0,0,0,0+1}*[*(lime) 1]{2+1,0}*[*(lime) 2]{0,2+1}
			\end{displaymath}

            \item Suppose that $\beta=(2,1^{q-2})$ and $q$ odd.
            Then $\lambda=(3,1^{n-2}),(2,2,1^{n-4}),(2^q,1^{p-q}).$ Since $\lambda/\alpha$ is vertical strip in all these cases (using $p\geq 15$ and $q\leq 5$ in the last case),
             we may replace $\beta$ by $(1^q)$.
            This completes the proof for this case.
    
            \item Suppose that $\beta=(1^q)$ and $q$ even.
            Then $\lambda$ must be equal to $(2,1^{n-2})$. In this case one can simply check directly, for example using the inner product of characters.
        \end{itemize}
        
    For the remaining, finitely many pairs $(p,q)$, one can directly compute using sage for example. 
    \end{proof}

We have a nice corollary which follows from the proof of Lemma~\ref{lemma:Yang_stareletov}.
        \begin{corollary}\label{corollary:Yang_stareletov}
            Let $\mu=(p,q)$ be a partition of $n$ with $p\geq 3$.
        Then
        \begin{equation}\label{equation_corollary:n_t_t}
            \sum\limits_{\substack{\alpha\vdash p,\beta\vdash q \\ \alpha\neq (p),(1^{p}) \\ f_{q}\geq s_\beta }}s_\alpha s_\beta\geq s_\lambda.
        \end{equation}
        holds for all partitions $\lambda\vdash n$ except when $\lambda=(3,3),(2,2,2)$ with $p=q=3$ and possibly when  $\lambda\in \{(n),(2,1^{n-2}),(1^n)\}$.
        \end{corollary}

    \begin{lemma}\label{lemma:eigen_value_-1}
        Let $(p,q)$ be a partition of $n$ with $ w_{(p,q)} $ having even order and $p\geq 6, q\geq 1$.
        Then $\ch\Ind_{C_{(p,q)}}^{ S_n} -1\geq s_\lambda$ for all partitions $\lambda$ of $n$ except when $\lambda$ equal to $(n)$ and possibly $(1^n)$. 
    \end{lemma}

    \begin{proof}
        Suppose that $w_{(p,q)}$ is an odd permutation. Then the theorem follows from
        Corollary~\ref{corollary:eigen_value_sgn}.
        So we may assume that $w_{(p,q)}$ is an even permutation, and thus both $p,q$ and $p+q=n$ are even.
        
        Since $g_{(p,q)}=\ch \Ind_{C_{(p,q)}}^{ S_n} -1\geq\ch\Ind_{C_{(p)}\times C_q}^{ S_n} -1\times1$ and $\ch\Ind_{C_{(n)}}^{ S_n} -1\geq s_\lambda$ for all partitions of $\lambda$ except when $\lambda=(n),(2,1^{n-2})$ by Theorem~\ref{theorem:Main_min}, we have  
        \begin{displaymath}
            g_{(p,q)}\geq \sum\limits_{\alpha\neq (p),(2,1^{p-2})}s_\alpha \sum\limits_{\alpha\neq (q-1,1),(1^{q})}s_\beta.
        \end{displaymath}

        Let $\lambda$ be a partition of $n$ and $\lambda\notin\{(n),(n-1,1),(2,1^{n-2}),(1^n)\}$.
        Using Lemmas~\ref{lemma:alphabeta},~\ref{lemma:choose-beta}, we get a pair of partitions $\alpha$ and $\beta$ such that $s_\alpha s_\beta\geq s_\lambda$ and $f_q\geq s_\beta.$
        If $g_p\geq s_\alpha$, then there is nothing to prove.
        Otherwise one of the following cases must occur.
        \subsubsection*{Case 1: $\alpha=(p)$.}
        Since $\lambda\neq (n)$ we may replace $\alpha$ by $(p-1,1)$ and choose $\beta$ as in the proof of \ref{lemma:alphabeta}. If $f_q\geq s_\beta$, we are done.
        Otherwise one of the following cases occurs.
        \begin{itemize}
            \item Suppose that $\beta=(q-1,1)$.
            
            Then $\lambda/\alpha$ contains $q-1$ columns, of which precisely one column contains two cells. The column of $\lambda/\alpha$ which contains two cells can be one of the first, the second or the $p$th column.
            
            If the first column of $\lambda/\alpha$ contains two cells, then incrementing the entries in the first column by $1$, allows us to replace $\beta$ by $(q-2,1,1)$ as shown in the example below. Note that $q\geq 4$.

                  	\begin{displaymath}
						\ydiagram{7,1}*[*(yellow) 1]{7+1,1+2,0}*[*(yellow) 2]{0,0,0}*[*(yellow) 3]{0,0,0,0}*[*(lime) 1]{0,0,0+1,0}*[*(lime) 2]{0,0,0,0+1}\longrightarrow
						\ydiagram{7,1}*[*(yellow) 1]{7+1,1+2,0}*[*(yellow) 2]{0,0,0}*[*(yellow) 3]{0,0,0,0}*[*(lime) 1]{0,0,0}*[*(lime) 2]{0,0,0+1}*[*(lime) 3]{0,0,0,0+1}
					\end{displaymath}

            If the second column of $\lambda/\alpha$ contains two cells, then incrementing the entry in the first column by $1$, allows us to replace $\beta$ by $(q-2,2)$ as shown in the example below.
                  	\begin{displaymath}
						\ydiagram{7,1}*[*(yellow) 1]{7+1,2+1,0+1}*[*(yellow) 2]{0,0,0}*[*(yellow) 3]{0,0,0,0}*[*(lime) 1]{0,1+1,0}*[*(lime) 2]{0,0,1+1}\longrightarrow
						\ydiagram{7,1}*[*(yellow) 1]{7+1,2+1,0}*[*(magenta) 2]{0,0,0+1}*[*(yellow) 3]{0,0,0,0}*[*(lime) 1]{0,1+1,0}*[*(lime) 2]{0,0,1+1}
					\end{displaymath}
                        
            If the $p$th column of $\lambda/\alpha$ contains two cells, then we may replace $\alpha$ by $(p-2,2)$ and $\beta$ by $(p-2,2)$ as shown in the example below.
                  	\begin{displaymath}
						\ydiagram{7,1}*[*(yellow) 1]{0,1+6,0}*[*(yellow) 2]{0,0,0}*[*(yellow) 3]{0,0,0,0}*[*(lime) 1]{7+1,0}*[*(lime) 2]{0,7+1}\longrightarrow
						\ydiagram{6,2}*[*(yellow) 1]{0,2+4,0}*[*(yellow) 2]{0,0,0}*[*(yellow) 3]{0,0,0,0}*[*(lime) 1]{7+1,0}*[*(lime) 2]{0,7+1}*[*(magenta) 1]{6+1}*[*(magenta) 2]{0,6+1}
					\end{displaymath}
            Note that in this case $p$ must be equal to $q$.
                
            \item Suppose that $\beta=(1^q)$ and $q$ even.
            Then $\lambda/\alpha$ has exactly one column and there exist a cell in the $p$th column of $\lambda/\alpha$. This implies there is no possible possible Young diagram except when $q=1$. Since $q$ is even, this case is also excluded.
        \end{itemize}
        \subsubsection*{Case 2: $\alpha=(2,1^{p-2})$.}

        Now we shall consider the following cases.
		\begin{itemize}
			\item Suppose $\lambda_2\geq 2$.
			We shall replace $\alpha$ by $(2,2,1^{p-4})$.
			Choose $\beta$ as in the proof of Lemma~\ref{lemma:alphabeta}.
			If $f_q\geq s_\beta$, then we are done.
			Otherwise one of the following cases must occur.
			\begin{itemize}
				\item $\beta=(q-1,1)$.
				In this case, $\lambda/\alpha$ has $q-1$ columns, with one column having two cells.
				The column with two cells has to be one of the first three columns.
				
				If the first or second column of $\lambda/\alpha$ has two cells, increment the entries of those two cells by $1$ to replace $\beta$ by $(q-2,1,1)$.
				Since $p\geq 4$, $\lambda/\alpha$ has at least one cell in the first row, so the resulting skew-tableau is an LR-tableau.

                    \begin{displaymath}
						\ydiagram{2,2,1,1}*[*(yellow) 1]{2+2,0,1+1,0,0,0,0,0,0}*[*(yellow) 2]{0,0,0,0,0,0,0,0,0}*[*(yellow) 3]{0,0,0,0}*[*(lime) 1]{0,0,0,0,0+1,0,0,0,0}*[*(lime) 2]{0,0,0,0,0,0+1}\longrightarrow
						\ydiagram{2,2,1,1}*[*(yellow) 1]{2+2,0,1+1,0,0,0,0,0,0}*[*(yellow) 2]{0,0,0,0,0,0,0,0,0}*[*(yellow) 3]{0,0,0,0}*[*(lime) 2]{0,0,0,0,0+1,0,0,0,0}*[*(lime) 3]{0,0,0,0,0,0+1} 
      \hspace{1.5cm}
                        \ydiagram{2,2,1,1}*[*(yellow) 1]{2+2,0,0,0,0+1,0,0,0}*[*(yellow) 2]{0,0,0,0,0,0,0,0,0}*[*(yellow) 3]{0,0,0,0}*[*(lime) 1]{0,0,1+1,0,0,0,0}*[*(lime) 2]{0,0,0,1+1}\longrightarrow
						\ydiagram{2,2,1,1}*[*(yellow) 1]{2+2,0,0,0,0+1,0,0,0,0}*[*(yellow) 2]{0,0,0,0,0,0,0,0,0}*[*(yellow) 3]{0,0,0,0}*[*(lime) 2]{0,0,1+1,0,0,0,0}*[*(lime) 3]{0,0,0,1+1}
					\end{displaymath}

				Suppose the third column of $\lambda/\alpha$ has two cells.
				Since $\lambda\supset (2,1^{p-2})$, the first column of $\lambda/\alpha$ must have exactly one cell.
				Changing the entry of this cell from $1$ to $3$ allows us to replace $\beta $ by  $(q-2,1,1)$.

                \begin{displaymath}    
                \ydiagram{2,2,1,1}*[*(yellow) 1]{3+1,0,1+1,0,0+1,0}*[*(yellow) 2]{0,0,0,0,0,0,0,0,0}*[*(yellow) 3]{0,0,0,0}*[*(lime) 1]{2+1,0,0,0,0,0}*[*(lime) 2]{0,2+1}\longrightarrow
				\ydiagram{2,2,1,1}*[*(yellow) 1]{3+1,0,1+1,0,0}*[*(yellow) 2]{0,0,0,0,0,0,0,0,0}*[*(yellow) 3]{0,0,0,0}*[*(lime) 1]{2+1,0,0,0,0,0}*[*(lime) 2]{0,2+1}*[*(magenta) 3]{0,0,0,0,0+1}
				\end{displaymath}

				\item Suppose that $\beta=(1^q)$.
				In this case, $\lambda$ must be equal to $(2,2,1^{n-4})$. Then we may choose $\alpha=(1^p)$ and $\beta=(2,1^{q-2})$ except when $q=2$. In this case, $\lambda=(2,2,1^{p-4})$. 
                One can verify that when $\lambda=(2,2,1^{p-4})$ the theorem is true.
			\end{itemize}
            \item Suppose $\lambda_2=1$ (so that $\lambda$ is a hook) and $\lambda_1\geq 3$.
			Replace $\alpha$ by $(3,1^{p-3})$ and choose $\beta$ according to the proof of Lemma~\ref{lemma:alphabeta}.
			If $f_q\geq s_\beta$, then we are done.
			Otherwise one of the following cases must occur.
			\begin{itemize}
				\item Suppose that $\beta = (q-1,1)$. In this case, $\lambda/\alpha$ has $q-1$ columns, one having two cells and the other having only one cell.
				Since $\lambda$ is a hook, only the first column of $\lambda/\alpha$ can have two cells.
				Incrementing the entries in the first column by one will allow us to replace $\beta$ to $(q-2,1,1)$.
				\item Suppose that $\beta = (1^q)$.
				In this case, $\lambda$ must be $(3,1^{n-3}).$
                We may replace $\alpha$ by $(1^p)$ and, $\beta$ by $(2,1^{q-2}).$
			\end{itemize}
   \end{itemize}
This completes the proof.  
\end{proof}

\begin{proof}[Proof of Theorem~\ref{Theorem:main}]

        Let $\mu\vdash n\geq 11$ such that all its parts divides $\mu_1$.
        We will determine, for which pair $(\lambda,\mu)$, the degree of the minimal polynomial of $\rho_\lambda(w_\mu)$ is not equal to $\mu_1$.
        It is suffices to check whether all $\mu_1$th roots of unity are eigenvalues of $\rho_\lambda(w_\mu)$.
        
        We know that $\rho_\lambda(w_\mu)$ has eigenvalue $1$ except for the pairs $(\lambda,\mu)$ given in Theorem~\ref{theorem:inv_vec_sym}.

        If the length of $\mu$ is $1$, then the theorem follows from Theorem~\ref{theorem:inv_vec_sym}.
        So let us assume that the length of $\mu$ is greater than or equal to $2$ and we may also assume that $\mu\neq (1^n)$.
        
        Let us consider a non-real $\mu_1$th root of unity (say $t$) and hence $\mu_1\geq 3$.
        Using Lemma~\ref{lemma:Yang_stareletov}, we have that $t$ is an eigenvalue of $\rho_\lambda (w_{(\mu_1,\mu_2)})$ except when $\lambda=(n),(1^n)$.
        
        Now we note that $\mu_1+\mu_2\geq 6$ or $\mu_i\leq 1$ for all $i\geq 2$ since $\mu_2|\mu_1$.
        
        In the former case, we can use Corollary~\ref{corollary:Yang_stareletov} with $p$ replaced by $\mu_1+\mu_2$ and $q$ replaced by $\mu_3$, and verify for the partitions listed as exceptions in Corollary~\ref{corollary:Yang_stareletov} directly. We conclude
        that $t$ is an eigenvalue of $\rho_\lambda (w_{(\mu_1,\mu_2,\mu_3)})$. Continuing this process with $p$ replaced by $\sum_{i=1}^k\mu_{i}$ and $q$ replaced by $\mu_{k+1}$ in Corollary~\ref{corollary:Yang_stareletov} and using induction we complete the proof in this case.
        
        For the latter case, we have $\mu_i\leq 1$ for all $i\geq 2$.
        Since $\mu_1|\mu_2$, we have $3\leq \mu_1\leq 5$.
        Hence $\mu=(\mu_1,1,1,1,\dotsc,1)$, and we see that $\rho_\lambda(\mu_1,1,1,1)$ has minimal polynomial $x^{\mu_1+3}-1$ for all partitions $\lambda$ of $\mu_1+3$ except when $\lambda=(\mu_1+3) \text{ or }(1^{\mu_1+3})$.
        Now we could replace $p$ by $\mu_1+3$ and $q$ by $1$ in Corollary~\ref{corollary:Yang_stareletov} and then use induction to complete the proof as in the previous case.
        

        Finally, let us consider the eigenvalue $-1$ in the case of $\mu_1$ even. We assumed that $\mu\neq (1^n)$ and the length of $\mu$ is greater than or equal to two.
        If $w_\mu$ is an odd permutation, then using Corollary~\ref{corollary:eigen_value_sgn}, we have that $-1$ is an eigen value of $\rho_\lambda(w_\mu)$. 
        Thus we may assume that $w_\mu$ is an even permutation, hence $w_{\tilde\mu}$ is an odd permutation where $\tilde\mu=(\mu_2,\mu_3,\dotsc)$.
        
        Suppose that $\mu_1=2$, hence $w_\mu$ is an involution.
        Then we can easily see that $-1$ must be an eigenvalue for all faithful representations of $ S_n$.
        The only non-faithful representations are precisely when $\lambda\in \{(n),(1^n), (2,2)\}$.
        Indeed, if $\rho_\lambda$ is not faithful, then the kernel $K$ forms a normal subgroup which is not equal to the trivial subgroup $\{e\}$. If $K=S_n$, then $\rho_\lambda$ is the trivial representation. Hence, in this case, $\lambda=(n)$.
        Otherwise, $K$ is a non-trivial normal subgroup of $S_n$.
        But the only non-trivial normal subgroups of $S_n$ are $A_n$ and, if $n=4$ we have one more, namely, the Klein-four subgroup $K_4$ which consists of all permutations of cycle type equal to $(2,2)$ and identity $\{e\}$.
        Therefore, if $K=A_n$, then we have only two possible irreducible representations for $S_n/A_n$.
        In this case, $\lambda$ must be equal either $(n)$ or $(1^n)$.
        If $n=4$ and $K=K_4$, then there are only three possible irreducible representations for $S_n/K_4$.
        In this case, $\lambda$ must be equal to the one of the elements in the set $\{(4),(2,2),(1^4)\}$.
        One can directly verify the theorem when $\lambda\in \{(n),(1^n),(2,2)\}$.
        
        Therefore, we may assume that $\mu_1\geq 4$.
        Suppose that the length of $\mu=2$. Then we may use Lemma~\ref{lemma:eigen_value_-1} and checking the remaining finitely many cases directly will yield the theorem.
        Now we may assume that the length of $\mu$ is at least three.
        Using Corollary~\ref{corollary:Yang_stareletov} with $p$ replaced by $\mu_1+\mu_2$ and $q$ replaced by $\mu_3$, we can conclude that the theorem holds when the length of the partition $\mu$ is three.

        Continuing this process with $p$ replaced by $\sum_{i=1}^k\mu_{i}$ and $q$ replaced by $\mu_{k+1}$ in Corollary~\ref{corollary:Yang_stareletov} and using induction, we complete the proof in this case.
        
        This completes the proof.
\end{proof}
Finally, we prove some applications of the Lemmas we have proved.
\begin{proof}[Proof of Theorem~\ref{theorem:eigen_value_negative_universal}]
    Let $n\geq 11$. For $n\leq 11$ we verify the theorem by direct computation. 
    Suppose that $w_\mu$ is an odd permutation. Then the theorem follows from Corollary~\ref{corollary:eigen_value_sgn}.
    Otherwise, at least two parts of $\mu$ is even, say $\mu_i,\mu_j$ for some $i>j$.
    
    Suppose that $\mu_1\geq 4$. Then we may consider the permutation $w_\nu$ where $\nu={(\mu_1,\mu_i)}$ (or equal to  $(\mu_1,\mu_j)$ if $i=1$). We shall work with $\nu=(\mu_1,\mu_i)$, since for the other case the proof is similar.
    Using Lemma~\ref{lemma:eigen_value_-1} and direct computation for the finitely many remaining cases, we get $g_\nu\geq s_\lambda$ for all partitions of $\lambda\vdash \mu_1+\mu_i$ except possibly when $\lambda=(\mu_1+\mu_i) \text{ or } (1^{\mu_1+\mu_j})$. Now we may use Corollary~\ref{corollary:Yang_stareletov} with $p$ replaced by $\sum\limits_{l=1}^k\mu_l$ and $q$ by $\mu_{k+1}$. The theorem follows by induction.

    Suppose that $\mu_1\leq 3$ and hence $\mu=(3^k,2^l,1^t)$ with $l\geq 2$.
    If $k=0$, then the theorem follows from Theorem~\ref{Theorem:main}.
    Assume that $k\geq 1$.
    Since $g_{(3,2,2)}\geq s_\lambda$ for all partitions $\lambda\vdash 7$ except when $\lambda=(7),(1^7),$
    we can use Corollary~\ref{corollary:Yang_stareletov} with $p$ replaced by $7$ and $q$ replaced by any part of $\mu$ other than $\mu_1=3,\mu_{k+1}=2,\mu_{k+2}=2$. Now the theorem follows by iterating this process using Corollary~\ref{corollary:Yang_stareletov}.
    This completes the proof.
\end{proof}

\section{Proof of theorems on alternating groups}\label{section:Proofs of theorems regarding Alternating groups}
\begin{proof}[Proof of Theorem~\ref{theorem:eigen_value_negative_universal_Alt}]
    Let $V$ be a non-trivial irreducible representation of $A_n$.
    If $V=V_\lambda$ with $\lambda\neq\lambda'$, then $\rho_\lambda(w_\mu)$ has eigenvalue $-1$ from Theorem~\ref{theorem:eigen_value_negative_universal}.
    Otherwise, $V=V_\lambda^\pm$ with $\lambda=\lambda'$.
    In this case, we have $\chi_\lambda^\pm(w_\mu^i)=\tfrac{\chi_\lambda(w_\mu^i)}{2}$.
    Computing the inner product of the characters $\Res_{C_\mu}^{A_n}\chi_\lambda^\pm,\delta$ where $\delta\in Irr(C_\mu)$  yields that $w_\mu$ has eigenvalue $-1$ in $V_\lambda^+$ and $V_\lambda^-$ if and only if it has eigenvalue $-1$ in $V_\lambda$. Now the theorem follows from Theorem~\ref{theorem:eigen_value_negative_universal}.
\end{proof}

\begin{proof}[Proof of Theorem~\ref{theorem:Main_2_min_alternating}]
One can prove this theorem independently by the arguments similar to those of the proof of Theorem~\ref{theorem:Main_min} about symmetric groups.
But for the sake of completeness we shall give a short proof which uses Theorem~\ref{theorem:Main_min} of symmetric groups.

We verify the theorem for $n\leq 25$ by direct computation or using sage.
Now we assume that $n\geq 27$ is an odd positive integer.
Let $(\rho,V)$ be an irreducible representation of the alternating group $A_n$.
If $V$ is not equal to both $V_{(\tfrac{n+1}{2},\tfrac{n-1}{2})}^+$ and $V_{(\tfrac{n+1}{2},\tfrac{n-1}{2})}^-$, then the theorem follows from Theorem~\ref{theorem:Main_min}.
Because the character value is $\rho(w_\mu^i)=\chi_\lambda(w_\mu^i)$ for all $i$ or $\rho(w_\mu^i)=\tfrac{\chi_\lambda(w_\mu^i)}{2}$ for all $i$.
In either case, the inner product of the characters yields that $t$ is an eigenvalue for $w_n^\pm$ in $V$ if and only if $t$ is an eigenvalue for $w_n$ in $V_\lambda$. The theorem follows from Theorem~\ref{theorem:Main_min} in this case.

Finally, let $V=V_{(\tfrac{n+1}{2},\tfrac{n-1}{2})}^\pm$.
Suppose that $p$ is prime. Then one can easily verify that the minimal polynomial of $\rho(w_n)$ is $x^n-1$.
Otherwise, let us write $n=mp$ where $p$ is the smallest prime dividing $n$, and hence $m\geq 5$.
Note that $C_m\wr C_p\subset A_n$.
We may choose $\mu^1=(m-2,1,1)$ and $\mu^2=\mu^3=\dotsc=\mu^p=(m)$. Then we see that 
$\Ind_{A_m^{\times p}}^{A_n} \chi_{\mu^1}\times\dotsc \times\chi_{\mu^p}\geq \chi_{\lambda}^\pm$.
Now the theorem follows using induction and Proposition~\ref{proposition:Giannelli}.
\end{proof}

\begin{proof}[Proof of Theorem~\ref{Theorem:Main_Alt}]
Let $\mu$ be a partition of $n\geq 26$ with all its parts divides $\mu^1$ such that $w_\mu$ is an even permutation.
If the length of $\mu$ is one, then the theorem follows from Theorem~\ref{theorem:Main_2_min_alternating}.
If $\mu\notin DOP_n$, then the theorem follows from Theorem~\ref{Theorem:main}.
Otherwise, $\mu\in DOP_n$. Suppose that $\lambda\neq \phi(\mu)$ where $\phi(\mu)$ is the folding of the partition $\mu$, then in this case the result follows from Theorem~\ref{Theorem:main}. 
Hence let $\lambda=\phi(\mu)$.
We have
\begin{align*}
\Ind_{C_\mu}^{A_n} \delta\geq \Ind_{C_{\mu_1}\times C_{\mu_2}\times \dotsc \times C_{\mu_k}}^{A_n}\tilde\delta\geq \Ind_{A_{\mu_1}\times A_{\mu_2}\times \dotsc \times A_{\mu_k}}^{A_n}\Ind_{C_{\mu_1}\times C_{\mu_2}\times \dotsc \times C_{\mu_k}}^{A_{\mu_1}\times A_{\mu_2}\times \dotsc \times A_{\mu_k}}\tilde\delta, 
\end{align*}
where $\delta$ is any linear character of $C_\mu$ and $\tilde\delta=\bar\delta\times1\times\dotsc\times1$ with $\bar\delta(w_{\mu_1}^i)=\delta(w_\mu^i)$ for all $i$.
Let $\nu^i=\phi(\mu_i)$ for $i=1,2,\dotsc,k$ where $k$ is the number of parts of $\mu$.
If $3$ or $5$ is not a part of $\mu$, then we are done, since 

\begin{align*}
\Ind_{C_\mu}^{A_n} \delta
&\geq \Ind_{A_{\mu_1}\times A_{\mu_2}\times \dotsc \times A_{\mu_k}}^{A_n}\Ind_{C_{\mu_1}\times C_{\mu_2}\times \dotsc \times C_{\mu_k}}^{A_{\mu_1}\times A_{\mu_2}\times \dotsc \times A_{\mu_k}}\tilde\delta\\
&\geq \Ind_{A_{\mu_1}\times A_{\mu_2}\times \dotsc \times A_{\mu_k}}^{A_n} \chi_{\nu^1}\times\dotsc\times\chi_{\nu^k}\\
&\geq \chi_{\phi(\mu)}
\end{align*}
In particular, $\Ind_{C_\mu}^{A_n} \delta \geq \chi_{\phi(\mu)}^+ \text{ and } \Ind_{C_\mu}^{A_n} \delta \geq \chi_{\phi(\mu)}^-$. We are done.

So let us assume that $3$ or $5$ is a part (say $\mu_j$) of $\mu$.
Note that either $3$ or $5$ is a part of $\mu$.
Then we may choose $\nu^{j+1}=(\tfrac{\mu_{j+1}+1}{2},2,1^{\tfrac{\mu_{j+1}-5}{2})}$ and 
if $\mu_j=3$, then $\nu^j= (3)$, otherwise $\nu^j=(4,1)$.
We let $\nu^t=\phi(\mu_t)$ for all $t\neq j,j+1$.
Similarly we have, $\Ind_{C_\mu}^{A_n} \delta\geq \chi_\lambda^+$ and $\Ind_{C_\mu}^{A_n} \delta\geq \chi_\lambda^-$.
This completes the proof.
\end{proof}
\subsection*{Acknowledgements}
I am deeply grateful to my guide, A.~Prasad, for his unwavering encouragement and many fruitful discussions.
We thank A.~Staroletov for pointing out the relevance of results of Giannelli and Law~\cite{Giannelli_law} which lead us to Proposition~\ref{proposition:Giannelli}.
We thank S.~Sundaram for her encouragement and fruitful discussions.
We also extend our gratitude to her for carefully reviewing a preliminary version of this article and suggesting numerous improvements that significantly enhanced the quality of this exposition.
We thank D.~Prasad and S.~Viswanath for their encouragement and fruitful discussions.
We thank R.~Kundu and Amrutha~P for fruitful discussions.

\bibliographystyle{abbrv}
\bibliography{refs}
\end{document}